\pdfoutput=1
\RequirePackage{ifpdf}
\ifpdf 
\documentclass[pdftex]{sigma}
\else
\documentclass{sigma}
\fi

\usepackage{amsxtra}
\usepackage{amscd}
\usepackage{eucal}
\usepackage{accents}
\usepackage{enumitem}

\usepackage{cancel}

\usepackage{ascmac}
\usepackage{fancybox}

\usepackage{mathtools}
\usepackage{tikz-cd}
\usepackage{mathtools}
\usepackage{pgf}
\usetikzlibrary{arrows, matrix}

\usepackage{here}
\usepackage{simplewick}

\numberwithin{equation}{section}
\newtheorem{thm}{Theorem}[section]
\newtheorem{prop}[thm]{Proposition}
\newtheorem{lem}[thm]{Lemma}
\newtheorem{cor}[thm]{Corollary}

\newcommand{\nn}{\nonumber}

\theoremstyle{definition}

\newtheorem{Example}[thm]{Example}
\newtheorem{rem}[thm]{Remark}

\newtheorem{Notation}[thm]{Notation}

\newcommand{\ket}[1]{{| #1 \rangle}}

\newcommand{\C}{{\mathbb C}}
\newcommand{\Z}{{\mathbb Z}}

\newcommand{\cC}{{\mathcal C}}

\newcommand{\E}{{\mathcal E}}

\newcommand{\cH}{\mathcal H}
\newcommand{\cK}{\mathcal{K}}

\newcommand{\GS}{\mathfrak{S}}
\newcommand{\gl}{\mathfrak{gl}}

\newcommand{\ve}{\varepsilon}

\newcommand{\bE}{\mathbf{E}}

\newcommand{\bG}{\mathbf{G}}
\newcommand{\bK}{\mathbf{K}}

\newcommand{\bu}{\mathbf{u}}
\newcommand{\bv}{\mathbf{v}}

\newcommand{\bx}{\mathbf{x}}
\newcommand{\by}{\mathbf{y}}
\newcommand{\bz}{\mathbf{z}}

\newcommand{\ssB}{\textsf{B}}

\newcommand{\ssD}{\textsf{D}}

\newcommand{\End}{\mathop{\rm End}}

\newcommand{\id}{{\rm id}}
\renewcommand{\Im}{\mathop{\rm Im}}

\newcommand{\Res}{\mathop{\rm res}}

\newcommand{\Sym}{\mathop{\mathrm{Sym}}}

\newcommand{\mc}{\mathcal}

\newcommand{\pdeg}{\mathop{\rm pdeg}}

\newcommand{\cont}[2]{\contraction[1ex]{}{#1}{}{#2} #1 #2}

\newcommand{\ft}{\tilde{f}}

\newcommand{\ovu}{\overline{u}}
\newcommand{\ovv}{\overline{v}}

\newcommand{\bbar}[1]{\bar{\bar{#1}}}

\begin{document}
\allowdisplaybreaks

\newcommand{\arXivNumber}{2112.14631}

\renewcommand{\PaperNumber}{051}

\FirstPageHeading

\ShortArticleName{Quantum Toroidal Comodule Algebra of Type $A_{n-1}$ and Integrals of Motion}

\ArticleName{Quantum Toroidal Comodule Algebra of Type $\boldsymbol{A_{n-1}}$\\ and Integrals of Motion}

\Author{Boris FEIGIN~$^{\rm ab}$, Michio JIMBO~$^{\rm c}$ and Evgeny MUKHIN~$^{\rm d}$}

\AuthorNameForHeading{B.~Feigin, M.~Jimbo and E.~Mukhin}

\Address{$^{\rm a)}$~National Research University Higher School of Economics,\\
\hphantom{$^{\rm a)}$}~20 Myasnitskaya Str., Moscow, 101000, Russia}
\Address{$^{\rm b)}$~Landau Institute for Theoretical Physics, 1a Akademika Semenova Ave.,\\
\hphantom{$^{\rm b)}$}~Chernogolovka, 142432, Russia}
\EmailD{\href{mailto:bfeigin@gmail.com}{bfeigin@gmail.com}}

\Address{$^{\rm c)}$~Department of Mathematics, Rikkyo University, Toshima-ku, Tokyo 171-8501, Japan}
\EmailD{\href{mailto:jimbomm@rikkyo.ac.jp}{jimbomm@rikkyo.ac.jp}}

\Address{$^{\rm d)}$~Department of Mathematics, Indiana University Purdue University Indianapolis,\\
\hphantom{$^{\rm d)}$}~402 N.~Blackford St., LD 270, Indianapolis, IN 46202, USA}
\EmailD{\href{mailto:emukhin@iupui.edu}{emukhin@iupui.edu}}

\ArticleDates{Received March 02, 2022, in final form June 27, 2022; Published online July 07, 2022}

\Abstract{We introduce an algebra $\mathcal{K}_n$ which has a structure of a left comodule over the quantum toroidal algebra of type $A_{n-1}$. Algebra $\mathcal{K}_n$ is a higher rank generalization of $\mathcal{K}_1$, which provides a uniform description of deformed $W$ algebras associated with Lie (super)al\-gebras of types BCD. We show that $\mathcal{K}_n$ possesses a family of commutative subalgebras.}

\Keywords{quantum toroidal algebras; comodule; integrals of motion}

\Classification{81R10; 81R12; 17B69; 17B80}

\section{Introduction}

Integrable systems originated in quantum toroidal algebras have been drawing a lot of attention in recent years,
see \cite{AO,FJMM1, FJM2,FKSW,FOS1, FOS2,GG,KZ, LV1, LV2},
to name a few.
In this paper we continue the study launched in~\cite{FJMV} concerning
deformed $W$ algebras and
a quantum toroidal version of integrable models with reflections in the spirit of \cite{Sk}.

The work \cite{FJMV} is based on a new algebra $\mc K_1$ generated by
current $E(z)$, Heisenberg half-cur\-rents~$K^{\pm}(z)$, and a central element $C$.
In various representations the current $E(z)$ recovers the fundamental current of 
known deformed $W$ algebras of types BCD as well as their supersymmetric analogs.
It was shown that $\mc K_1$ has a family of commuting elements $\{\mathbf{I}_N\}_{N=1}^\infty$ called integrals of motion.
The element $\mathbf{I}_N$ is written as an $N$-fold integral of the product
$\mathbf{E}(z_N)\cdots \mathbf{E}(z_1)$ with the explicit elliptic kernel of~\cite{Ru}.
Here $\mathbf{E}(z)$ is the dressed current of the form $\mathbf{E}(z)=E(z)\mathbf{K}^+(z)^{-1}$,
where $\mathbf{K}^+(z)$ is given in terms of $K^+(z)$. In the representations corresponding to deformed~$W$ algebras,
the integrals of motion $\mathbf{I}_N$ are deformations of local integrals of motion in the $W$ algebras, see \cite{FF}.

In the present paper we generalize this construction by introducing algebras $\mc K_n$ for $n>1$.
The algebra $\mc K_n$ depends on parameters $q_1,q_2=q^2$, and is generated by currents
$E_1(z),\dots,E_n(z)$, Heisenberg half-currents $K_1^{\pm}(z),\dots, K_n^\pm(z)$, and a central element $C$.
We show that it has the following key properties:\pagebreak
\begin{itemize}\itemsep=1pt
\item Algebra $\mc K_n$ has
essentially the same size as the Borel subalgebra of quantum toroidal $\mathfrak{gl}_n$ algebra $\mc E_n$;
see Remark \ref{rem:size}.
\item For $m<n$, algebra $\mc K_m$ can be obtained as a subalgebra of $\mc K_n$ using the fusion procedure
in the same way as $\mc E_m$ is found as a subalgebra of $\mc E_n$ in \cite{FJMM3};
see Proposition \ref{prop:fusion}.
\item Algebra $\mc K_n$ is a left comodule over quantum toroidal algebra $\mc E_n$;
see Theorem \ref{thm:comod}.
\item Algebra $\mc K_n$ has two ``boundary" modules $\mathbb{F}^D$ of level $C=q^{-1}$ and $\mathbb{F}^B$ of level $C=q^{1/2}$;
see Lemmas \ref{lem:typeCD} and~\ref{lem:typeB}.
\item Algebra $\mc K_n$ has a family of commuting integrals of motion $\{\mathbf{I}_N\}_{N=1}^\infty$,
such that $\mathbf{I}_N$ is an $nN$-fold integral of a current given by an ordered product of the form
$\prod_{1\le i \le n}^{\curvearrowleft}\prod_{1\le a\le N}^{\curvearrowleft}\mathbf{E}_i(z_a)$ together
with the kernel of \cite{KS}; see \eqref{IM1} and Theorem \ref{thm:main}.
Here the dressed currents have the form $\mathbf{E}_i(z)=\mathbf{K}^-_i(z)^{-1} E_i(z)$, where $\mathbf{K}^-_i(z)$ is given in terms of
$K^-_i(z)$.\footnote{We change the convention slightly from \cite{FJMV}.}
\end{itemize}

Here we should make a disclaimer: the relations of $\mc K_n$ involve infinite sums, so that
an appropriate completion is necessary. The same is true with the fusion procedure.
Also the comodule map is defined at the level of a class of representations called admissible ones.

These properties ensure that $\mc K_n$ has a rich representation theory.
At least, we have representations of the form $V\otimes W$, where $W$ is one of the two boundary modules and
$V$ is any admissible $\mc E_n$ module.
Here by an admissible module we mean a homogeneously graded module with the degree bounded from above;
for example any tensor product of Fock modules (or, more generally, modules with highest weight
with respect to the ``rotated'' set of generators) is admissible.
Then, under a mild restriction on parameters $q_1$, $q_2$ and the level $C$,
in any $\mc K_n$ module we have an integrable system given by action of $\{\mathbf{I}_N\}$. This is the principal result of this paper.
For $n=1$, this system is related to the system with reflections, see~\cite[Section 5.1]{FJMV}, and the Bethe ansatz for
it is studied in \cite{LV2} in the setting of affine Yangians.
Although the expressions of $\{\mathbf{I}_N\}$ are non-local, we shall sometimes refer to them as
``local'' integrals of motion, since in the case $n=1$ they are deformations of local ones as already mentioned.

We do not completely understand the nature of algebras~$\mc K_n$.
The~$\mc K_n$ algebras we discuss here are of A type.
Similar algebras can be defined in the same way for arbitrary simply laced types,
though the construction of integrals of motion is more challenging.
At first glance, algebra~$\mc K_n$ bears resemblance to $\imath$quantum groups;
formula \eqref{delta-E} for comodule structure is reminiscent of the Chevalley generators~\cite[formula~(5.1)]{Ko},
and the defining rela\-tions~\mbox{\eqref{KK1}--\eqref{Serre_3}} of~$\cK_n$
look similar to the Drinfeld realization given in~\cite[Theorem~5.1]{LW}.
However, the details seem to be very different.
For example, the relations in $\cK_n$ are invariant under sca\-ling $z\to cz$ but
such homogeneity does not exist in those of \cite{LW}.
The current appearing in the residues in the right-hand side of~(5.5) in~\cite{LW}
is commutative, whereas our current $K_i^-(z)K_i^+\big(C^2z\big)$ in~\eqref{EE} is not.
The $\imath$quantum groups are coideals in the corresponding quantum groups, while
we are not aware of inclusion of~$\mc K_n$ to $\mc E_n$.
We do not know if $\imath$quantum groups have interesting families of integrals of motion in general.
Natural candidates which could be compared to~$\mc K_n$ are the reflection algebras and, in particular,
quantum twisted Yangians, see~\cite{MRS}.
The quantum twisted Yangians are coideal subalgebras behind integrable systems with reflections.
However, the exact relation to~$\mc K_n$ algebras is unclear. The quantum twisted Yangians are given in terms of the
$R$-matrix realization (as opposed to the Drinfeld-type rea\-li\-za\-tion of~$\mc K_n$)
and the formulas for the generators of the coideal are quadratic in generators of quantum group.\looseness=1

In quantum toroidal algebras $\mc E_n$, the integrals of motion originate from the standard construction of transfer matrices
and depend on $n$ arbitrary parameters.
They commute for a simple reason following from the Yang--Baxter equation.
Then a computation provides the expression for the integrals of motion in the form of integrals.
The integrals of motion in the $\mc K_n$ algebra are given virtually by the same formula as in $\mc E_n$,
except that one of the parameters (the elliptic nome) is fixed by the central element $C$.
But we do not have any general reason for their commutativity; we simply check it directly.

It is known that in type A there is a natural duality of such models of $(\mathfrak{gl}_m, \mathfrak{gl}_{n})$ type, see~\cite{FJM3}.
Under that duality, the algebras $\mc E_m$ and $\mc E_{n}$ act on the same space in such a way that the screening operators
for one algebra are given in terms of generating currents of the other.
Furthermore the integrals of motion of both algebras commute with each other.
For $\mc K_1$ acting in a representation, we have a system of screening operators, see~\cite{FJMV}.
We expect that the algebras generated by the corresponding screening currents also possess a commutative family
which should be called ``non-local" integrals of motion, and that ``local'' and ``non-local'' integrals of motion commute.

Similarly, we expect that algebras $\mc K_n$ acting in representations commute with a family of screening operators.
Moreover the algebra generated by the screening currents does not depend on $n$.
For the case of the tensor product of $\ell$ Fock $\mc E_n$ modules with a boundary
$\mc K_n$ module, we anticipate $\ell+1$ screening currents each given as a sums of $n$ vertex operators.
The cor\-res\-pon\-ding ``non-local'' integrals of motion should commute with the ``local''
integrals of motion $\{\mathbf{I}_N\}$ constructed in this paper.
We plan to address this issue in future publications.

The text is organized as follows.

In Section \ref{sec:prel} we introduce our convention about the quantum toroidal algebra $\E_n$ of type $\gl_n$.
In Section \ref{sec:Kn} we define algebra $\cK_n$. Working with admissible representations, we rephrase the Serre relations as
zero conditions (called wheel conditions) on matrix coefficients of products of generating currents.
We show that on admissible representations,
$\cK_n$ may be viewed as a left comodule over $\E_n$.
We use this fact to construct representations starting with elementary ones.
In Section \ref{sec:IM} we introduce integrals of motion. They are integrals of products of generating currents of $\cK_n$
obtained by a dressing procedure.
The kernel functions entering integrals of motion are precisely the same as those used for those of $\E_{n}$.
We then prove that integrals of motion mutually commute by direct computation.

For simplicity of presentation, we treat the case $n\ge3$ in the body of the text.
We shall mention the modifications necessary for $n=1,2$ in Section \ref{sec:n12}.

In Section \ref{sec:fusion} we show that,
for any $k=1,\dots,n-1$, algebra $\cK_n=\cK_n(q_1,q_2,q_3)$ with parameters $q_1$, $q_2$, $q_3$
contains mutually commuting subalgebras $\cK_k(\bar q_1,\bar q_2,\bar q_3)$ and $\cK_{n-k}(\bbar q_1,\bbar q_2,\bbar q_3)$
with appropriate parameters $\bar q_i$, $\bbar q_i$.

In Appendix \ref{sec:identity}, we give a proof of the theta function identities used in the commutativity of integrals of motion.

\begin{Notation}
Throughout the text we fix parameters $q^{1/2},d^{1/2}\in\C^{\times}$ and define
\begin{align*}
&q_1=q^{-1}d,\qquad q_2=q^2,\qquad q_3=q^{-1}d^{-1},
\end{align*}
so that $q_1q_2q_3=1$.
We assume that $q_1^iq_2^j=1$ for $i,j\in\Z$ implies $i=j=0$.
\end{Notation}

For a positive integer $N$, we write $\delta^{(N)}_{i,k}=1$ if $i\equiv k\bmod N$ and $\delta^{(N)}_{i,k}=0$ otherwise.

We use the standard symbols for infinite products
\begin{align*}
&(z_1,\dots,z_r;p)_\infty=\prod_{i=1}^r\prod_{k=0}^\infty\big(1-z_i p^k\big),
\qquad
\Theta_p(z)=\big(z,pz^{-1},p;p\big)_\infty.
\end{align*}
We use also the formal series $\delta(z)=\sum_{k\in\Z}z^k$.\pagebreak

\section{Preliminaries}
\label{sec:prel}
In this section we fix our convention regarding the quantum toroidal algebra $\E_n$ of type $\gl_n$.

\subsection[Quantum toroidal algebra E\_n]{Quantum toroidal algebra $\boldsymbol{\E_n}$}

Fix a positive integer $n\ge3$.
We define $g_{i,j}(z,w)$, $G_{i,j}(x)$ for $i,j\in\Z/n\Z$ as follows\vspace{-1ex}
\begin{gather*}
g_{i,i}(z,w)=z-q^2w, \qquad G_{i,i}(x)=q^2\frac{1-q^{-2}x}{1-q^2x},
\\
g_{i,i\pm1}(z,w)=z-q^{-1}d^{\pm1}w, \qquad G_{i,i\pm 1}(x)=q^{-1}\frac{1-qd^{\pm1}x}{1-q^{-1}d^{\pm1}x},
\\
g_{i,j}(z,w)=z-w, \qquad G_{i,j}(x)=1\qquad \text{if}\quad j\not \equiv i,i\pm1.
\end{gather*}
Here $G_{i,j}(x)$ should be understood as power series expansion in $x$.
As rational function we have
\begin{gather*}
G_{i,j}(w/z)=-\frac{g_{j,i}(w,z)}{d_{i,j}g_{i,j}(z,w)}=G_{j,i}(z/w)^{-1},
\end{gather*}
where $d_{i,j}=d^{\pm 1}$, $i\equiv j\pm1$, and $d_{i,j}=1$ otherwise.

Let $P$ be a free $\Z$ module with basis $\{\ve_i\mid i\in\Z/n\Z\}$, equipped with a symmetric bilinear form $(~,~)\colon P\times P\to\Z$
such that $\ve_i$ are orthonormal. We write $\bar\alpha_i=\ve_{i-1}-\ve_{i}$.

The quantum toroidal algebra $\E_n$ of type $\gl_n$ is a unital associative algebra generated by
\begin{gather*}
e_{i,k},\quad f_{i,k},\quad h_{i,r},\quad i\in\Z/n\Z,\ k\in \Z,\ r\in\Z\backslash\{0\},
\qquad q^{h},\quad h\in P, \qquad C^{\pm1}.
\end{gather*}
In terms of the generating series
\begin{gather*}
e_i(z)=\sum_{k\in\Z}e_{i,k}z^{-k},\qquad
f_i(z)=\sum_{k\in\Z}f_{i,k}z^{-k},
\\
\psi^{\pm}_i(z)=\psi_{i,0}^{\pm} \exp\bigg({\pm}\big(q-q^{-1}\big)\sum_{r>0}h_{i,\pm r}z^{\mp r}\bigg),\qquad \psi_{i,0}^\pm=q^{\pm\bar\alpha_i},
\end{gather*}
the defining relations read as follows
\begin{gather*}
q^hq^{h'}=q^{h+h'},\quad h,h'\in P,\qquad q^0=1,\qquad \text{$C$ is central},
\\
q^he_i(z)q^{-h}=q^{(h,\bar \alpha_i)}e_i(z),\quad
q^hf_i(z)q^{-h}=q^{-(h,\bar \alpha_i)}f_i(z),\quad
q^h\psi^\pm_i(z)q^{-h}=\psi_i^\pm(z),\quad h\in P,
\\[1ex]
\psi^\pm_i(z)\psi^\pm_j (w) = \psi^\pm_j(w)\psi^\pm_i (z),
\\
\psi^+_i(z)\psi^-_j (w)=\psi^-_j(w)\psi^+_i (z) G_{i,j}(C w/z)^{-1}G_{i,j}\big(C^{-1} w/z\big),
\\[1ex]
\psi_i^+(z)e_j(w)=e_j(w)\psi_i^+(z) G_{i,j}\big(C^{-1}w/z\big),\qquad
e_j(w)\psi_i^-(z)=\psi_i^-(z)e_j(w) G_{j,i}(z/w),
\\
\psi_i^+(z)f_j(w)=f_j(w)\psi_i^+(z) G_{i,j}(w/z)^{-1},\qquad
f_j(w)\psi_i^-(z)=\psi_i^-(z)f_j(w) G_{j,i}(C z/w)^{-1},
\\[1ex]
[e_i(z),f_j(w)]=\frac{\delta_{i,j}}{q-q^{-1}}
(\delta\bigl(C{w}/{z}\bigr)\psi_i^+(w)
-\delta\bigl(C{z}/{w}\bigr)\psi_i^-(z)),
\\
[e_i(z),e_j(w)]=0, \qquad [f_i(z),f_j(w)]=0, \qquad i\not\equiv j,j\pm1,
\\
d_{i,j}g_{i,j}(z,w)e_i(z)e_j(w)+g_{j,i}(w,z)e_j(w)e_i(z)=0,
\\
d_{j,i}g_{j,i}(w,z)f_i(z)f_j(w)+g_{i,j}(z,w)f_j(w)f_i(z)=0,
\\[1ex]
\Sym_{{z_1,z_2}}[e_i(z_1),[e_i(z_2),e_{i\pm1}(w)]_q]_{q^{-1}}=0,
\\
\Sym_{{z_1,z_2}}[f_i(z_1),[f_i(z_2),f_{i\pm1}(w)]_q]_{q^{-1}}=0.
\end{gather*}
Here $[A,B]_p=AB-pBA$, and the symbol $\mathrm{Sym}$ stands for symmetrization:
\begin{gather*}
\Sym_{x_1,\dots,x_N}\ f(x_1,\dots,x_N) =\frac{1}{N!}
\sum_{\sigma\in\GS_N} f(x_{\sigma(1)},\dots,x_{\sigma{(N)}}).
\end{gather*}

The relations above imply in particular
\begin{gather}
[h_{i,r},h_{j,s}]=\delta_{r+s,0} \cdot a_{i,j}(r)\frac{C^r-C^{-r}}{q-q^{-1}},\qquad r,s\neq 0,\nn
\\
a_{i,j}(r)=\frac{1}{r}\frac{q^r-q^{-r}}{q-q^{-1}}
\big((q^r+q^{-r})\delta^{(n)}_{i,j}-d^r \delta^{(n)}_{i,j-1}-d^{-r}\delta^{(n)}_{i,j+1}\big).
\label{E-hh}
\end{gather}

Algebra $\E_n$ has a $\Z$ grading which we call homogeneous grading, given by
\begin{gather*}
\deg x_{i,r}=r,\quad x=e,f,h,\qquad \deg q^h=\deg C=0.
\end{gather*}

\subsection{Coproduct}
Quite generally, let $\mc A$ be a $\Z$ graded algebra with a central element $C$.
The completion of $\mc A$ in the positive direction is the algebra $\tilde{\mc A}$, linearly spanned by products
of series of the form $\sum_{k=M}^\infty f_k g_k$, where $M\in \Z$, $f_k,g_k\in\mc A$ and $\deg g_k=k$.
We call an $\mc A$ module $V$ admissible if $C$ is diagonalizable,
and if $V$ is $\Z$ graded with finite dimensional components of degree bounded from above:
$V=\oplus_{k=-\infty}^{N} V_k$, where $V_k=\{v\in V\mid \deg v=k\}$,
$\dim V_k<\infty$.
The completion $\tilde {\mc A}$ acts on all admissible modules.

Let $\E_n\tilde \otimes \E_n$ be the tensor algebra $\E_n\otimes \E_n$ completed in the positive direction.
We use the topological coproduct $\Delta\colon \E_n\to \E_n\tilde \otimes \E_n$
\begin{gather*}
\Delta e_i(z)=e_i(C_2z) \otimes \psi^{-}_i(z)+ 1 \otimes e_i(z),
\\
\Delta f_i(z)=f_i(z)\otimes 1 +\psi^{+}_i(z)\otimes f_i(C_1z),
\\
\Delta \psi^{+}_i(z)=\psi^{+}_i(z)\otimes \psi^{+}_i(C_1z),
\\
\Delta \psi^{-}_i(z)=\psi^{-}_i(C_2z)\otimes \psi^{-}_i(z),
\\
\Delta x=x\otimes x,\qquad x=q^{h}, C.
\end{gather*}
Here $C_1=C\otimes 1$ and $C_2=1\otimes C$.

Later on we shall use the elements in the completed algebra $\tilde\E_n$
\begin{gather*}
\ft_i(z)=S^{-1}\bigl(f_i(z)\bigr) =-f_i\big(C^{-1}z\big)\psi_i^+\big(C^{-1}z\big)^{-1},
\end{gather*}
where $S$ denotes the antipode.
The $\tilde{f}_i(z)$'s satisfy the same quadratic relations and the Serre relations as do $e_i(z)$'s. In addition
\begin{gather}
d_{j,i}g_{j,i}(w,z)e_i(z)\ft_j(w)+g_{i,j}(z,w)\ft_j(w)e_i(z)\nn
\\ \qquad
{}=-\delta_{i,j}\frac{g_{i,i}(w,z)}{q-q^{-1}}
\bigl(\delta(w/z)
-\delta\big(C^2 z/w\big)\psi_i^-(z){\psi_i^+(C z)}^{-1}\bigr).\label{ftil2}
\end{gather}

\subsection{Fock modules}

It is well known \cite{Sa} that $\E_n$ has admissible representations given in terms of vertex operators.
Consider a vector space with basis $\ket{\mathbf{m}}$ labeled by $n$ tuples of integers
$\mathbf{m}=(m_0,m_1,\dots,m_{n-1})\allowbreak\in\Z^n$. We define linear operators
${\rm e}^{\pm\epsilon_i}$ and $\partial_i$, $0\le i\le n-1$, by 
\begin{gather*}
{\rm e}^{\pm\epsilon_i}\ket{\mathbf{m}}=(-1)^{\sum_{s=0}^{i-1}m_s}\ket{\mathbf{m}\pm \mathbf{1}_i},\qquad
\partial_i\ket{\mathbf{m}}=m_i \ket{\mathbf{m}},
\end{gather*}
with $\mathbf{1}_i=(\delta_{i,j})_{0\le j\le n-1}$. We have
${\rm e}^{\epsilon_i}{\rm e}^{\epsilon_j}=-{\rm e}^{\epsilon_j}{\rm e}^{\epsilon_i}$ for $i\neq j$.
We set
\begin{gather*}
\mathbb{F}=\C[\{h_{i,-r}\}_{r>0,0\le i\le n-1}]\otimes \bigg(\bigoplus_{\mathbf{m}\in\Z^n}\C\ket{\mathbf{m}}\bigg).
\end{gather*}
We write operators $h_{i,\pm r}\otimes\id$, $\id\otimes {\rm e}^{\pm\epsilon_i}$ and $\id\otimes \partial_i$
simply as $h_{i,\pm r}$, ${\rm e}^{\pm\epsilon_i}$ and $\partial_i$. We extend the range of the suffix $i$ by periodicity,
e.g., $\partial_{i+n}=\partial_i$. Define a vertex operator which acts on $\mathbb{F}$
\begin{gather*}
V_i(z)={\rm e}^{-\epsilon_{i}}{\rm e}^{\epsilon_{i-1}}z^{\partial_{i-1}-\partial_i}d^{(\partial_{i-1}+\partial_i)/2}
\exp\bigg(\sum_{r>0}\frac{q^{-r}}{[r]}h_{i,-r}z^r\bigg)
\exp\bigg({-}\sum_{r>0}\frac{1}{[r]}h_{i,r}z^{-r}\bigg).
\end{gather*}

For $u\in\C^\times$, the following assignment gives $\mathbb{F}$ a structure of an admissible
$\E_n$ module
denoted by $\mathbb{F}(u)$,
\begin{gather*}
e_i(z)\mapsto u^{-\delta_{i,0}}z V_i(z),\qquad
\tilde{f}_i(z)\mapsto -q^{-1}u^{\delta_{i,0}}z\, {:}V_i(z)^{-1}{:},
\\
h_{i,r}\mapsto h_{i,r},\quad r\neq0,\qquad
q^{\varepsilon_i}\mapsto q^{\partial_i},\qquad C\mapsto q.
\end{gather*}
Here we use the usual normal ordering symbol to bring ${\rm e}^{\pm\epsilon_i}$, $h_{i,-r}$, $r>0$,
to the left and $\partial_i$, $h_{i,r}$, $r>0$, to the right, except that the order between
${\rm e}^{\pm\epsilon_i}$'s is kept unchanged.
The $\E_n$ module~$\mathbb{F}(u)$ is a direct sum of irreducible submodules, see \cite[Lemma 3.2]{FJM3}.

\section[Comodule algebra K\_n]{Comodule algebra $\boldsymbol{\cK_n}$}\label{sec:Kn}

In this section we introduce an algebra $\cK_n$,
and show that it has the structure of an $\E_n$ comodule in the sense to be made precise below.
Algebra $\cK_1$ was first introduced and studied in \cite{FJMV} in order to
give a uniform description of deformed $W$ algebras associated with simple Lie (super)algebras of types BCD.

\subsection[Algebra K\_n]{Algebra $\boldsymbol{\cK_n}$}
We introduce an algebra $\cK_n$ through generating series
\begin{gather*}
E_i(z)=\sum_{k\in\Z} E_{i,k}z^{-k},\qquad
K^{\pm}_i(z)=K^{\pm}_{i,0}\exp\biggl(\pm\big(q-q^{-1}\big)\sum_{\pm r>0} H_{i,r}z^{-r}\biggr),
\qquad i\in\Z/n\Z,
\end{gather*}
and an invertible central element $C$.
The defining relations read as follows
\begin{gather}
K_i^{\pm}(z)K_j^\pm(w)=K_j^\pm(w)K_i^{\pm}(z),
\label{KK1}
\\
K_i^+(z)K^-_j(w)=K^-_j(w)K_i^+(z)G_{i,j}(w/z)G_{i,j}\big(C^{2}w/z\big),
\label{KK2}
\\
K_i^{+}(z)E_j(w)=E_j(w)K_i^{+}(z)G_{i,j}(w/z),
\quad
E_j(w)K_i^{-}(z)=K_i^{-}(z)E_j(w)G_{j,i}(z/w),\label{KK3}
\\
d_{i,j}g_{i,j}(z,w)E_i(z)E_j(w)+g_{j,i}(w,z)E_j(w)E_i(z)
\label{EE}
\\ \quad
{}=\!\frac{\delta_{i,j}}{q\!-\!q^{-1}}
\big(g_{i,i}(z,w)\delta\bigl(C^2z/w\bigr)K^-_i(z)K^+_i\big(C^2z\big)
\!+\!g_{i,i}(w,z)\delta\bigl(C^2w/z\bigr)K^-_i(w)K^+_i\big(C^2w\big)\big),\nn
\\
E_i(z)E_j(w)=E_j(w)E_i(z),\qquad i\not\equiv j,\, j\pm1.\label{EE2}
\end{gather}
In addition, we impose the Serre relations
\begin{align}
\mathop{\Sym}_{z_1,z_2}\,[E_i(z_1),[E_i(z_2),E_{i\pm 1}(w)]_q]_{q^{-1}}
=
-\mathop{\Sym}_{z_1,z_2}&\bigg\{\bigg(\frac{1}{1-q^{-1}d^{\mp1}z_1/w}
+\frac{q^{-1}d^{\pm1}w/z_2}{1-q^{-1}d^{\pm1}w/z_2}\bigg)\nn
\\
&\times \delta\big(C^2z_1/z_2\big)K_i^-(z_1)E_{i\pm 1}(w)K_i^+(z_2)\bigg\}.\label{Serre_3}
\end{align}
In the right-hand side, the rational functions of the form $1/(1-a z_1/w)$ or $1/(1-b w/z_2)$
stand for their expansions in non-negative powers of $z_1/w$ or $w/z_2$, respectively.

In the right-hand sides of \eqref{EE} and \eqref{Serre_3}, the Fourier coefficients of $K_i^-(z)K_i^+\big(C^{2}z\big)$ are infinite series.
Therefore the relations require some justification. We proceed as follows.
We~define a $\Z$ grading of generators by
\begin{gather*}
\deg E_{i,k}=k,\qquad \deg H_{i,\pm r}=\pm r,\qquad \deg K_{i,0}^{\pm}=\deg C=0.
\end{gather*}
We call it the homogeneous grading.
Consider the free algebra $\mc A$ generated by $E_{i,k}$, $H_{i,\pm r}$, $K^\pm_{i,0}$,~$C^{\pm1}$.
Let $\widetilde {\mc A}$ be the completion of $\mc A$ with respect to the grading in the positive direction.
Then we define $\cK_n$ to be the graded quotient algebra of $\widetilde {\mc A}$ under the relations \eqref{KK1}--\eqref{Serre_3}.

The quadratic relations \eqref{EE} imply
\begin{align}
d_{i,j}\hat{g}_{i,j}(z,w)E_i(z)E_j(w)+\hat{g}_{j,i}(w,z)E_j(w)E_i(z)=0,
\label{ghatEE}
\end{align}
where
\begin{gather*}
\hat{g}_{i,j}(z,w)=
\begin{cases}
\big(z-C^2w\big)\big(z-C^{-2}w\big)\big(z-q^2w\big), &i=j,
\\
g_{i,j}(z,w), &i\neq j.
\end{cases}
\end{gather*}
The factor $\big(z-C^2w\big)\big(z-C^{-2}w\big)$ for $i=j$ is chosen to kill the delta functions in the right-hand side of~\eqref{EE}.

\begin{rem}\label{rem:size}
Algebra $\cK_n$ has a filtration defined by
$\pdeg E_i(z)=1$, $\pdeg K_i^\pm(z)=0$ for all~$i$, which we call
the principal filtration.
In the associated graded algebra, the subalgebra gene\-ra\-ted by $\{E_i(z)\}$ is isomorphic to
the subalgebra of the quantum toroidal algebra $\E_n$ generated by~$\{e_i(z)\}$.

The relations for $K^\pm_i(z)$ are slightly different from those of $\psi^{\pm}_i(z)$ in $\E_n$,
since the right-hand side of \eqref{KK2} contains a product of $G_{i,j}$'s as opposed to the ratio.
In particular, while the elements $\psi^\pm_{i,0}\in \E_n$ are inverse to each other,
the elements $K_{i,0}^{\pm}\in\cK_n$ are not mutually commutative,
\begin{align*}
&K_{i,0}^+ K_{j,0}^-=q^{2(\bar \alpha_i,\bar\alpha_j)} K_{j,0}^-K_{i,0}^+ .
\end{align*}
We have the relations
\begin{align*}
[H_{i,r},H_{i,s}]=-\delta_{r+s,0} \cdot a_{i,j}(r)\frac{1+C^{2r}}{q-q^{-1}},\qquad r>0,
\end{align*}
which are to be compared with \eqref{E-hh}.
In the limit $q\to 1$ with $C=q^c$, algebra $\E_n$ reduces to the enveloping algebra of a Lie algebra
of matrix valued difference operators, see~\cite[Section 3.7]{FJMM3}.
In~contrast, algebra $\cK_n$ does not seem to have a reasonable limit due to the presence of the sum~$1+C^{2r}$.
\end{rem}

\subsection{Wheel conditions}

All representations of $\E_n$ or $\cK_n$ considered in this paper are admissible representations.
In this setting, we rewrite the formal series relations of generating currents
in the language of matrix coefficients.

Let $V=\!\oplus_{k=-\infty}^{k_0} \!V_k$ be a graded vector space with degrees bounded from above and \mbox{$\dim V_k\!<\!\infty$}.
We denote by $V^*=\oplus_{k=-\infty}^{k_0} V^*_k$ the restricted dual space.
Suppose that we are given a set of formal series $\big\{E_i(z)=\sum_{k\in\Z}E_{i,k}z^{-k}\mid i\in\Z/n\Z\big\}$, such that
$E_{i,k}\in\End V$, $E_{i,k}V_l\subset V_{k+l}$,
sa\-tisfying the relations
\eqref{EE2} and \eqref{ghatEE} with some $C\in \C^\times$.
We assume that the zeros of $\hat{g}_{i,i}(z,w)$ are distinct:
\begin{gather*}
C^2\neq \pm 1, q^{\pm2}.
\end{gather*}
It follows from \eqref{EE2} and \eqref{ghatEE} that all matrix coefficients
\begin{gather*}
\langle w, E_{i_1}(z_1)\cdots E_{i_N}(z_N) v\rangle,\qquad v\in V,\quad w\in V^*,
\end{gather*}
converge in the region $|z_1|\gg\cdots\gg|z_N|$
to rational functions which have at most simple poles at $\hat{g}_{i_r,i_s}(z_{r},z_s)=0$, where
$r<s$ and $i_r\equiv i_s,i_s\pm1$.

Quite generally, let $A_1(z)$, $A_2(z)$ be operator-valued
formal series and $r_{i,j}(x)$ be rational functions.
We shall say that an exchange relation
\begin{gather*}
r_{1,2}(z_2/z_1)A_1(z_1)A_2(z_2)=r_{2,1}(z_1/z_2)A_2(z_2)A_1(z_1)
\end{gather*}
holds as rational functions,
if an arbitrary matrix coefficient of each side converges (in regions $|z_1|\gg |z_2|$ and $|z_2|\gg |z_1|$ respectively)
to the same rational function.
Similarly we shall use the term ``exchange relation as meromorphic functions''.

Equations \eqref{EE2} and \eqref{ghatEE} imply that we have relations as rational functions
\begin{gather}
\lambda^0_{i,j}(w/z)E_i(z)E_j(w)=\lambda^0_{j,i}(z/w)E_j(w)E_i(z),
\label{f0EE}
\end{gather}
where we set
\begin{gather*}
\lambda_{i,i}^0(x)=\frac{1-C^2x}{1-x}\frac{1-C^{-2}x}{1-x}\frac{1-q^2x}{1-x},
\\
\lambda_{i,i\pm1}^0(x)=d^{\mp1/2}\frac{1-q^{-1}d^{\pm1}x}{1-x},
\end{gather*}
and $\lambda^0_{i,j}(x)=1$ for $j\not\equiv i,i\pm1$.

Set
\begin{gather}
E_{i_1,\dots,i_N}(z_1,\dots, z_N)=\prod_{1\le r<s\le N}\lambda^0_{i_r,i_s}(z_{s}/z_{r}) E_{i_1}(z_1)\cdots E_{i_N}(z_N).
\label{E^N}
\end{gather}
It follows from \eqref{f0EE} that
\begin{gather}
E_{i_{\sigma(1)},\dots,i_{\sigma(N)}}(z_{\sigma(1)},\dots,z_{\sigma(N)}) =E_{i_1,\dots,i_N}(z_1,\dots,z_N),\qquad
\sigma\in \mathfrak{S}_N.
\label{E-symm}
\end{gather}
In particular \eqref{E^N} is symmetric with respect to $z_r$ and $z_s$ when they have the same ``color'' $i_r=i_s$.
Moreover all matrix coefficients of \eqref{E^N} have the form
\begin{gather}
\frac{P(z_1,\dots,z_N)}{\prod_{r<s}(z_{r}-z_s)^{N_{r,s}}},
\qquad
P(z_1,\dots,z_N)\in\C\big[z_1^{\pm1},\dots,z_N^{\pm1}\big],
\label{Laurent}
\end{gather}
where $N_{r,s}=2$ if $i_r=i_s$, $N_{r,s}=1$ if $i_r\equiv i_s\pm 1$, and $N_{r,s}=0$ otherwise.
Note that, when $i_r=i_s$,
third order poles are absent due to the symmetry in variables of the same color.

The following is easy to see.
\begin{prop}\label{prop:EErel}
Under the conditions \eqref{EE2} and \eqref{ghatEE}, the condition
\eqref{EE} holds if and only if
\begin{gather}
E_{i,i}\big(z,C^2z\big)=\frac{1}{q-q^{-1}}\frac{1+C^2}{1-C^2}\frac{1-q^2C^2}{1-C^2}\, K^-_i(z)K_i^+\big(C^2z\big).\label{quad-res}
\end{gather}
\end{prop}

We have also
\begin{prop}\label{prop:wheel}
Under the conditions \eqref{EE2} and \eqref{ghatEE},
the Serre relations \eqref{Serre_3} are satisfied
if and only if
the following zero conditions hold:
\begin{gather}
E_{i,i,i\pm1}\big(z,q^2 z,q d^{\mp1}z\big)=0,\label{wheel1}
\\
E_{i,i,i\pm1}\big(z,C^2 z, C^2q^{-1}d^{\mp1}z\big)=0.\label{wheel2}
\end{gather}
If it is the case then we have also\vspace{-.5ex}
\begin{gather}
E_{i,i,i\pm1}\big(z,C^{-2}z,C^{-2}q d^{\mp1}z\big)=0.
\label{wheel3}
\end{gather}
\end{prop}

We call \eqref{wheel1}--\eqref{wheel3} ``wheel conditions''.
\begin{proof}
The left-hand side and the right-hand side of the Serre relations \eqref{Serre_3}
read respecti\-vely~as
\begin{align*}
{\rm LHS}=\!{}&\Sym_{z_1,z_2}\bigl(E_i(z_1)E_i(z_2)E_{i+1}(w)\!-\!\big(q\!+\!q^{-1}\!\big)E_i(z_1)E_{i+1}(w)E_i(z_2) \!+\!E_{i+1}(w)E_i(z_1)E_i(z_2)\bigr),
\\
{\rm RHS}=\!{}&-\Sym_{z_1,z_2}\biggl(-\frac{q_3^{-1}w/z_1}{1-q_3^{-1} w/z_1}+\frac{q_1w/z_2}{1-q_1w/z_2}
+\delta(q_3z_1/w)\biggr)\delta\big(C^2z_1/z_2\big)K_i^-(z_1)
\\
&\times E_{i+1}(w)K^+_i(z_2).
\end{align*}
The case involving $E_i(z_1)$, $E_i(z_2)$ and
$E_{i-1}(w)$ can be treated by interchanging $q_3$ with $q_1$.

Using the definition \eqref{E^N}, {\rm LHS} can be rewritten as
\begin{align*}
{\rm LHS}={}&E_{i,i,i+1}(z_1,z_2,w)\bigg(
\frac{1}{\lambda^0_{i,i}\big(\frac{z_2}{z_1}\big)\lambda^0_{i,i+1}\big(\frac{w}{z_1}\big)\lambda^0_{i,i+1}\big(\frac{w}{z_2}\big)}
-\frac{q+q^{-1}}{\lambda^0_{i,i}\big(\frac{z_2}{z_1}\big)\lambda^0_{i,i+1}\big(\frac{w}{z_1}\big) \underline{\lambda^0_{i+1,i}\big(\frac{z_2}{w}\big)}}
\\
&+\frac{1}{\lambda^0_{i,i}\big(\frac{z_2}{z_1}\big) \underline{\lambda^0_{i+1,i}\big(\frac{z_1}{w}\big)\lambda^0_{i+1,i}\big(\frac{z_2}{w}\big)}}
+\frac{1}{\underline{\lambda^0_{i,i}\big(\frac{z_1}{z_2}\big)}\lambda^0_{i,i+1} \big(\frac{w}{z_1}\big)\lambda^0_{i,i+1}\big(\frac{w}{z_2}\big)}
\\
&-\frac{q+q^{-1}}{\underline{\lambda^0_{i,i}\big(\frac{z_1}{z_2}\big)\lambda^0_{i+1,i} \big(\frac{z_1}{w}\big)}\lambda^0_{i,i+1}\big(\frac{w}{z_2}\big)}
+\frac{1}{\underline{\lambda^0_{i,i}(\frac{z_1}{z_2})\lambda^0_{i+1,i} \big(\frac{z_1}{w}\big)\lambda^0_{i+1,i}\big(\frac{z_2}{w}\big)}}\bigg).
\end{align*}
The matrix coefficients of $E_{i,i,i+1}(z_1,z_2,w)$ are Laurent polynomials up to poles on $z_1=z_2$ or~$z_i=w$,
which are cancelled by the zeros of $1/\lambda^0_{i,j}$.

This is a sum of six formal series, each expanded in a different region.
In order to bring them to expansions in a common region, we use the identity $1/(1-z)=\delta(z)-z^{-1}/\big(1-z^{-1}\big)$,
where $1/(1-z)$ stands for its expansion in non-negative powers of $z$, while $z^{-1}/\big(1-z^{-1}\big)$ means the one in negative powers of $z$.
In the underlined factors, we substitute
\begin{gather*}
\frac{1}{\lambda^0_{i+1,i}\big(\frac{z}{w}\big)}=\text{(expansion in $|z|\gg|w|$)}
+d^{-1/2}\big(1-q_3^{-1}\big)\delta(q_3{z}/{w}),
\\
\frac{1}{\lambda^0_{i,i}\big(\frac{z_1}{z_2}\big)}=\text{(expansion in $|z_1|\gg|z_2|$)}
-\frac{q^{-2}\big(1-q^2\big)^3}{\big(1-C^{-2}q^2\big)\big(1-C^2q^2\big)}\delta\big(q^{2}{z_1}/{z_2}\big)
\\ \hphantom{\frac{1}{\lambda^0_{i,i}(\frac{z_1}{z_2})}=}
{}+\frac{1-C^2}{1+C^{2}}\frac{1-C^{2}}{1-q^2C^{2}}\delta\big(C^{2}{z_2}/{z_1}\big)
+
\frac{1-C^{-2}}{1+C^{-2}}\frac{1-C^{-2}}{1-q^2C^{-2}}\delta\big(C^{-2}{z_2}/{z_1}\big)
\end{gather*}
and bring all terms into the sum of their expansions in $|z_1|\gg|z_2|\gg|w|$
and additional delta functions.
We then compare the coefficients of each product of delta functions.

Up to the symmetry $z_1\leftrightarrow z_2$, there are the following cases to consider:
\begin{align*}
(1)\quad& 1,&(2)\quad& \delta(q_3 z_1/w), &(3)\quad& \delta\big(q^2z_1/z_2\big), \\
(4)\quad& \delta\big(C^2z_2/z_1\big),&(5)\quad& \delta(q_3 z_1/w)\delta\big(q^2z_1/z_2\big), &(6)\quad& \delta(q_3 z_1/w)\delta\big(C^2z_2/z_1\big),
\\
(7)\quad& \delta(q_3 z_1/w)\delta\big(C^{-2}z_2/z_1\big), &(8)\quad& \delta(q_3 z_1/w) \delta(q_3 z_2/w).
\end{align*}

In cases (1), (2) and (3), the coefficients vanish due to the identities of rational functions
\begin{align*}
0={}&\Sym_{z_1,z_2}\bigg\{
\frac{1}{\lambda^0_{i,i}\big(\frac{z_2}{z_1}\big)\lambda^0_{i,i+1}\big(\frac{w}{z_1}\big)\lambda^0_{i,i+1}\big(\frac{w}{z_2}\big)}
-\frac{q+q^{-1}}{\lambda^0_{i,i}\big(\frac{z_2}{z_1}\big)\lambda^0_{i,i+1}\big(\frac{w}{z_1}\big)
\lambda^0_{i+1,i}\big(\frac{z_2}{w}\big)}
\\
&+\frac{1}{\lambda^0_{i,i}\big(\frac{z_2}{z_1}\big)\lambda^0_{i+1,i}\big(\frac{z_1}{w}\big)\lambda^0_{i+1,i}\big(\frac{z_2}{w}\big)}
\bigg\},
\\
0={}&\frac{1}{\lambda^0_{i,i}(z_2/z_1)}\frac{1}{\lambda^0_{i+1,i}(z_2/w)}
\!-\!\big(q\!+\!q^{-1}\big)\frac{1}{\lambda_{i,i}^0(z_1/z_2)}\frac{1}{\lambda^0_{i,i+1}(w/z_2)}
\!+\!\frac{1}{\lambda^0_{i,i}(z_1/z_2)}\frac{1}{\lambda^0_{i+1,i}(z_2/w)},
\\
0={}&
\frac{1}{\lambda^0_{i,i+1}(w/z_1)}\frac{1}{\lambda^0_{i,i+1}(w/z_2)}
-\big(q+q^{-1}\big)\frac{1}{\lambda^0_{i+1,i}(z_1/w)}\frac{1}{\lambda^0_{i,i+1}(w/z_2)}
\\
&+\frac{1}{\lambda^0_{i+1,i}(z_1/w)}\frac{1}{\lambda^0_{i+1,i}(z_2/w)}.
\end{align*}
In case (8) the coefficient vanishes on $z_1=z_2$ due to the triple zero of $\lambda^0_{i,i}(z_1/z_2)^{-1}$ there.

In case (4), the relevant terms are the 3 terms containing ${\lambda^0_{i,i}(z_1/z_2)}^{-1}$.
Using the relation~\eqref{quad-res}, we find that the coefficient of $\delta\big(C^2z_2/z_1\big)$ is
\begin{gather*}
\frac{1-C^2}{1+C^2}\frac{1-C^2}{1-q^2C^2}
E_{i,i}\big(C^2z_2,z_2\big)E_{i+1}(w) \lambda^0_{i,i+1}(w/z_1)\lambda^0_{i,i+1}(w/z_2)
\\ \qquad
{}\times \bigg(
\frac{1}{\lambda^0_{i,i+1}(w/z_1)}\frac{1}{\lambda^0_{i,i+1}(w/z_2)}
-\big(q+q^{-1}\big)\frac{1}{\lambda^0_{i+1,i}(z_1/w)}\frac{1}{\lambda^0_{i,i+1}(w/z_2)}
\\ \qquad\hphantom{\times \bigg(}
+\frac{1}{\lambda^0_{i+1,i}(z_1/w)}\frac{1}{\lambda^0_{i+1,i}(z_2/w)}\bigg)
\\ \qquad
{}=-\frac{d(w/z_1-q^2w/z_2)}{\big(1-q_3^{-1}w/z_1\big)\big(1-q_3^{-1}w/z_2\big)}
K_i^-(z_2)K^+_i\big(C^2z_2\big)E_{i+1}(w),
\end{gather*}
which coincides with the corresponding term in {\rm RHS}.

Likewise, in case (7) the coefficient of $\delta\big(C^2z_2/z_1\big)\delta\big(q_3^{-1}w/z_1\big)$ comes from
\begin{gather*}
E_{i,i,i+1}(z_1,z_2,w)\frac{1}{\lambda^0_{i,i}(z_1/z_2)}\frac{1}{\lambda^0_{i+1,i}(z_2/w)}
 \bigl({-}\big(q+q^{-1}\big)\lambda^0_{i,i+1}(w/z_2)+\lambda^0_{i+1,i}(z_2/w)\bigr).
\end{gather*}
Computing similarly we find that it matches with the corresponding term in {\rm RHS}.

It remains to consider (5) and (6). The coefficients of the product of delta functions are proportional (with non-zero multipliers)
to the left-hand side of \eqref{wheel1} and \eqref{wheel2}, respectively.
Since there are no terms that come from {\rm RHS},
we have shown that the Serre relations are equivalent to \eqref{wheel1} and \eqref{wheel2}.

Instead of rewriting terms into expansions in $|z_1|\gg|z_2|\gg|w|$,
one can equally well proceed to the opposite region $|z_1|\ll |z_2|\ll |w|$. Computing in the same way we obtain
 \eqref{wheel3} in place of \eqref{wheel2}.
\end{proof}

\subsection{Comodule structure}

In this section we show that, when we restrict to admissible modules, $\cK_n$ may be viewed as a~left comodule of
$\E_n$. To be precise we prove the following.\footnote{Formula for $\Delta$ is slightly changed from the one used in \cite{FJMV}.
The $\E_1$ comodule structure discussed there should be understood at the level of admissible representations.}

\begin{thm}\label{thm:comod}
Let $V$ be an admissible $\E_n$ module and $W$ an admissible $\cK_n$ module, such that~$C$ acts as scalar $C_1$, $C_2$ respectively. We set $C_{12}=C_1C_2$ and assume that
\begin{gather*}
C_{2}^2\neq \pm1, q^{\pm 2},\qquad
C_{12}^2\neq \pm1, q^{\pm 2}.
\end{gather*}
Then
$V\otimes W$ is given the structure of an admissible $\cK_n$ module by the action of the following currents:
\begin{gather}
\Delta E_i(z)=e_i(C_2z)\otimes K^-_i(z) + 1\otimes E_i(z) +\ft_i\big(C_2^{-1}z\big)\otimes K^+_i(z),
\label{delta-E}
\\
\Delta K_i^+(z)= \psi_i^+\big(C_1^{-1}C_2^{-1}z\big)^{-1} \otimes K_i^+(z),
\label{delta-Kp}
\\
\Delta K_i^-(z)= \psi_i^{-}(C_2 z) \otimes K_i^-(z),
\label{delta-Km}
\\
\Delta C=C_{12}.
\label{delta-C}
\end{gather}

Moreover the following ``coassociativity'' and ``counit property'' hold:
\begin{gather*}
(V_1\otimes V_2)\otimes W=V_1\otimes (V_2\otimes W),\\ \C\otimes W=W.
\end{gather*}
Here $V_1$, $V_2$ are admissible $\E_n$ modules and $\C$ denotes the trivial $\E_n$ module.
\end{thm}

\begin{proof}
On each vector of an admissible module, the currents of $\E_n$ or $\cK_n$
comprise only a finite number of negative powers in $z$.
Hence each Fourier coefficient of \eqref{delta-E}--\eqref{delta-Km} is a well-defined operator on $V\otimes W$.

It is easy to verify the relations \eqref{KK1}--\eqref{KK3} and \eqref{EE2}
by direct calculation. With the aid of~\eqref{ftil2}, one can also check \eqref{EE} for $i\neq j$.

Let us check \eqref{EE} for $i=j$.
We shall write $A_1(z,w)\equiv A_2(z,w)$ if each matrix coefficient of $A_1(z,w)-A_2(z,w)$ is a Laurent polynomial.
Then we have on $V$
\begin{gather}
g_{i,j}(z,w)e_i(z)e_j(w)\equiv 0,\qquad
g_{i,j}(z,w)\ft_i(z)\ft_j(w)\equiv 0,\label{singE1}
\\
e_i(z)\ft_j(w)\equiv \frac{\delta_{i,j}}{q-q^{-1}}\biggl(
{-}\frac{1}{1-w/z}+\frac{1}{1-C^{-2}w/z}\psi_i^-(z) \psi_i^+\big(C^{-1}w\big)^{-1}\biggr),
\label{singE3}
\end{gather}
and on $W$
\begin{gather}
g_{i,j}(z,w)E_i(z)E_j(w)\nn
\\ \qquad
{}\equiv
\frac{\delta_{i,j}}{q-q^{-1}}\bigg(
\frac{z-q^2w}{1-C^{-2}w/z}K_i^-(z)K_i^+(w)
+\frac{w-q^2z}{1-C^{2}w/z}K_i^-(w)K_i^+(z)\bigg).
\label{singK}
\end{gather}
Using these we obtain
\begin{gather*}
\big(q-q^{-1}\big)\big(z-q^2w\big)\Delta E_i(z)\Delta E_i(w)
\\ \qquad
{}\equiv\big(q-q^{-1}\big) \big(z-q^2w\big)\bigl(
e_i(C_2z)\ft_i\big(C_2^{-1}w\big)\otimes K_i^-(z)K_i^+(w)
\\ \qquad \hphantom{\equiv}
{}+\ft_i\big(C_2^{-1}z\big)e_i(C_2w) \otimes K_i^+(z)K_i^-(w)
+1\otimes E_i(z)E_i(w)\bigr)
\\ \qquad
{}\equiv -\biggl(\frac{z-q^2w}{1-C_2^{-2}w/z} -\frac{z-q^2w}{1-C_{12}^{-2}w/z}\psi^-_i(C_2z)\psi_i^+\big(C_{12}^{-1}w\big)^{-1}\biggr)
\otimes K_i^-(z)K_i^+(w)
\\ \qquad \hphantom{\equiv}
{}+\bigg(\frac{w-q^2z}{1-C_2^{-2}z/w}-\frac{w-q^2z}{1-C_{12}^{-2}z/w} \psi^-_i(C_2w)\psi_i^+\big(C_{12}^{-1}z\big)^{-1}\bigg) \otimes K_i^-(w)K_i^+(z)
\\ \qquad \hphantom{\equiv}
{}+\frac{z-q^2w}{1-C_2^{-2}w/z}1\otimes K_i^-(z)K_i^+(w)+\frac{w-q^2z}{1-C_2^{2}w/z}1\otimes K_i^-(w)K_i^+(z)
\\ \qquad
{}\equiv \frac{z-q^2w}{1-C_{12}^{-2}w/z}\psi_i^-(C_2z)\psi_i^+\big(C_{12}^{-1}w\big)^{-1}\otimes K_i^-(z)K_i^+(w)
\\ \qquad \hphantom{\equiv}
{}- \frac{w-q^2z}{1-C_{12}^{-2}z/w}\psi_i^-(C_2w)\psi_i^+\big(C_{12}^{-1}z\big)^{-1}\otimes K_i^-(w)K_i^+(z).
\end{gather*}
Interchanging $z$ with $w$ and summing these, we arrive at \eqref{EE} for $i=j$.

Next let us verify the Serre relations \eqref{Serre_3}. By Proposition \ref{prop:wheel}, it suffices to show
that the current $\Delta E_{i,i,i\pm1}(z_1,z_1,w)$ satisfies the
wheel conditions \eqref{wheel1}--\eqref{wheel2}. It is comprised of~27~terms which we group together according to the principal grading
in the first component, defined by
$\pdeg e_i(z)=1$, $\pdeg f_i(z)=-1$, $\pdeg \psi^\pm_i(z)=\pdeg C=0$.
For the terms of principal degrees~$\pm3$ and $\pm2$, the wheel conditions are easily checked.

Let us verify the case of principal degree $1$. There are 6 terms coming from
$\Delta E_{i,i,i\pm1}(z_1,z_2,w)$:
\begin{gather*}
A_1=\phi\, e_i(C_2z_1)e_i(C_2z_2)\ft_{i\pm1}\big(C_2^{-1}w\big)\otimes K_i^-(z_1) K_i^-(z_2) K_{i\pm1}^+(w),
\\
A_2=\phi\, e_i(C_2z_1)e_{i\pm1}(C_2 w)\ft_i\big(C_2^{-1}z_2\big)\otimes K_i^-(z_1)K_{i\pm1}^-(w) K_i^+(z_2) G_{i,i\pm1}(w/z_2),
\\
A_3=\phi\,e_i(C_2z_2)e_{i\pm1}(C_2w)\ft_i\big(C_2^{-1}z_1\big)\otimes K_i^-(z_2) K_{i\pm1}^-(w)K_i^+(z_1) G_{i,i}(z_2/z_1)G_{i,i\pm1}(w/z_1),
\\
A_4=\phi\,e_i(C_2z_1)\otimes K_i^-(z_1)E_i(z_2)E_{i\pm1}(w),
\\
A_5=\phi\,e_i(C_2z_2)\otimes K^-_i(z_2)E_i(z_1)E_{i\pm1}(w)G_{i,i}(z_2/z_1),
\\
A_6=\phi\,e_{i\pm1}(C_2w)\otimes K^-_{i\pm1}(w)E_i(z_1)E_i(z_2)G_{i,i\pm1}(w/z_1)G_{i,i\pm1}(w/z_2),
\end{gather*}
where
\begin{gather*}
\phi=\frac{z_1-C_{12}^2 z_2}{z_1-z_2}\frac{z_1-C_{12}^{-2}z_2}{z_1-z_2}\frac{z_1-q^2z_2}{z_1-z_2}
\frac{z_1-q^{-1}d^{\pm1}w}{z_1-w}\frac{z_2-q^{-1}d^{\pm1}w}{z_1-w}.
\end{gather*}
From \eqref{singE1}--\eqref{singK}, we see that the sum $\sum_{i=1}^6A_i$ has no poles other than $z_1=z_2$, $z_i=w$.
In~fact, the only possible poles at $z_2=C_2^2 z_1$ (resp. $z_2=C_2^{-2}z_1$) arise from $A_2$ and $A_6$
(resp.,~$A_3$ with $A_6$), and they cancel by virtue of \eqref{singE3} and \eqref{singK}.

Let us check \eqref{wheel1}.
Under the specialization $z_2=q^2z_1$, $w=q d^{\mp1}z_1=q^{-1} d^{\mp1}z_2$, each term has a zero due to
the vanishing factors
\begin{gather*}
\big(z_1-q^2z_2\big)G_{i,i}(z_2/z_1),\qquad
z_1-q^{-1}d^{\pm1}w,\qquad
\big(z_2-q^{-1}d^{\pm1}w\big)G_{i,i\pm1}(w/z_2).
\end{gather*}

To check \eqref{wheel2}, set $z_2=C_{12}^2z_1$.
All $A_i$'s vanish at $z_2=C_{12}^2z_1$ with the exception of $A_2$.
Since the latter has the factor $\big(z_2-q^{-1}d^{\pm1}w\big)G_{i,i\pm1}(w/z_2)$,
it also vanishes by setting further $w=q^{-1} d^{\mp1}z_2$.

In the case of principal degree $0$, there are 7 terms:
\begin{gather*}
B_0=\phi\,1\otimes E_i(z_1)E_i(z_2)E_{i\pm1}(w),
\\
B_1=\phi\,e_i(C_2z_1)\ft_i\big(C_2^{-1}z_2\big)\otimes K_i^-(z_1)E_{i\pm1}(w)K^+_i(z_2) G_{i,i\pm1}(w/z_2),
\\
B_2=\phi\,e_i(C_2z_2)\ft_i\big(C_2^{-1}z_1\big)\otimes K^-_i(z_2)E_{i\pm1}(w)K^+_i(z_1) G_{i,i}(z_2/z_1)G_{i,i\pm1}(w/z_1),
\\
B_3=\phi\,e_i(C_2z_1)\ft_{i\pm1}\big(C_2^{-1}w\big)\otimes K_i^-(z_1)E_i(z_2)K^+_{i\pm1}(w),
\\
B_4=\phi\,e_i(C_2z_2)\ft_{i\pm1}\big(C_2^{-1}w\big)\otimes K_i^-(z_2)E_i(z_1)K^+_{i\pm1}(w) G_{i,i}(z_2/z_1),
\\
B_5=\phi\,e_{i\pm1}(C_2w)\ft_i\big(C_2^{-1}z_2\big)\otimes K^-_{i\pm1}(w)E_i(z_1)K^+_i(z_2) G_{i,i\pm1}(w/z_1)G_{i,i\pm1}(w/z_2),
\\
B_6=\phi\,e_{i\pm1}(C_2w)\ft_i\big(C_2^{-1}z_1\big)\otimes K^-_{i\pm1}(w)E_i(z_2)K^+_i(z_1)
G_{i,i}(z_2/z_1) G_{i,i\pm1}(w/z_1)G_{i,i\pm1}(w/z_2).
\end{gather*}
Again possible poles at $z_2=C_2^{\pm2}z_1$ are cancelled among $B_0$, $B_1$, $B_2$ and the sum $\sum_{i=0}^6 B_i$ has no poles
other than $z_1=z_2$, $z_i=w$.

Consider $z_2=q^2z_1$, $w=q d^{\mp1}z_1=q^{-1} d^{\mp1}z_2$. By the assumption on $C_2^2$, $B_0$ vanishes.
The rest of the terms are zero for the same reason as in principal degree $1$.
Next let $z_2=C_{12}^2z_1$. The only term which may survive is $B_1$.
Due to the factor $\big(z_2-q^{-1}d^{\pm1}w\big)G_{i,i\pm1}(w/z_2)$
it also vanishes by setting further $w=q^{-1} d^{\mp1}z_2$.

The case of principal degree $-1$ being similar to principal degree $1$, we omit the details.
Finally the coassociativity and counit property can be readily checked by using \eqref{delta-E}--\eqref{delta-C}.
\end{proof}

\subsection{Representations}

Theorem \ref{thm:comod} allows us to construct a large family of admissible representations of $\cK_n$ starting from simple ones.

Algebra $\cK_n$ has a Heisenberg subalgebra $\cH_n$
generated by $H_{i,r}$, $i\in\Z/n\Z$, $r\neq0$, $K^{\pm}_{i,0}$, and~$C^{\pm1}$.
Let $\mathbb{F}^{\ssD}$ be an irreducible representation of $\cH_n$ on which $C$ acts as $q^{-1}$.

\begin{lem}\label{lem:typeCD}
With the following action, $\mathbb{F}^{\ssD}$ is an admissible $\cK_n$ module:
\begin{gather*}
E_i(z)\mapsto 0,\qquad K_i^\pm(z)\mapsto K_i^\pm(z),\qquad
C\mapsto q^{-1}.
\end{gather*}
\end{lem}

\begin{proof}
The defining relations are obviously satisfied except \eqref{EE}. To see the latter
it suffices to note that $g_{i,i}\big(z,C^{2}z\big)=0$ when $C=q^{-1}$.
\end{proof}

In the next case we need to adjoin $\big(K^{\pm}_{i,0}\big)^{1/2}$ to $\cH_n$ and $\cK_n$. We define $\tilde{K}_i^{\pm}(z)$ by
\begin{gather*}
\tilde{K}_i^{\pm}(z)\tilde{K}_i^{\pm}(q z)=K^\pm_i(z).
\end{gather*}
Let $\mathbb{F}^{\ssB}$ be an irreducible representation of $\cH_n$ on which $C$ acts as $q^{1/2}$.

\begin{lem}\label{lem:typeB}
With the following action, $\mathbb{F}^{\ssB}$ is an admissible $\cK_n$ module:
\begin{gather*}
E_i(z)\mapsto k\tilde{K}^-_i(z)\tilde{K}^+_i(q z),\qquad
K_i^\pm(z)\mapsto K_i^\pm(z),\qquad
C\mapsto q^{1/2},
\end{gather*}
where $k=1/(q^{1/2}-q^{-1/2})$.
\end{lem}
\begin{proof}
We can write $G_{i,j}(x)=\tilde{G}_{i,j}(x)\tilde{G}_{i,j}(q^{-1}x)$, where
\begin{gather*}
\tilde{G}_{i,i}(x)=q\frac{1-q^{-1}x}{1-x}\frac{1-qx}{1-q^2x},\qquad
\tilde{G}_{i,i\pm 1}(x)=q^{-1/2}\frac{1-qd^{\pm1}x}{1-d^{\pm1}x},
\end{gather*}
and $\tilde{G}_{i,j}(x)=1$ otherwise.
The assertion can be checked by noting
\begin{gather*}
\tilde{K}^+_i(z)\tilde{K}_j^-(w)=\tilde{K}_j^-(w)\tilde{K}^+_i(z)\tilde{G}_{i,j}(w/z).
\end{gather*}
We use also
\begin{gather*}
E_{i,i}(z,w)\mapsto
q k^2\tilde{K}^-_i(z)\tilde{K}^-_i(w)\tilde{K}^+_i(qz)\tilde{K}^+_i(qw)
\frac{1-q^{-2}w/z}{1-w/z}\frac{1-q^{2}w/z}{1-w/z}\,
\end{gather*}
and Proposition \ref{prop:EErel} to check the coefficients of delta functions in \eqref{EE},
and use Proposition \ref{prop:wheel} along with
\begin{align*}
E_{i,i,i\pm1}(z_1, z_2, w) \mapsto {}&
k^3\tilde{K}^-_i(z_1)\tilde{K}^-_i(z_2)\tilde{K}^-_{i\pm1}(w)
\tilde{K}^+_i(qz_1)\tilde{K}^+_i(qz_2)\tilde{K}^+_{i\pm1}(qw)
\\ \qquad
&\times \frac{1-q^{-2}z_2/z_1}{1-z_2/z_1}\frac{1-q^{2}z_2/z_1}{1-z_2/z_1}
\frac{1-d^{\pm1}w/z_1}{1-w/z_1}\frac{1-d^{\pm1}w/z_2}{1-w/z_2}\,
\end{align*}
for the Serre relations \eqref{Serre_3}.
\end{proof}

Together with Theorem \ref{thm:comod}, the above lemmas imply
\begin{cor}
Let $V$ be an admissible $\E_n$ module, where $C$ acts as scalar $C_1$ with $q^{-2}C_1^2\neq \pm1,q^{\pm2}$.
Then $V\otimes \mathbb{F}^{\ssD}$ is an admissible $\cK_n$ module with the action
\begin{gather*}
\Delta E_i(z)=e_i\big(q^{-1}z\big)\otimes K_i^-(z)+ \ft_i(qz)\otimes K^+_i(z),
\\
\Delta K_i^+(z)=\psi_i^+\big(q C_1^{-1}z\big)^{-1}\otimes K_i^+(z),
\qquad
\Delta K_i^-(z)=\psi_i^-\big(q^{-1}z\big)\otimes K_i^-(z),
\\
\Delta C= C_1q^{-1}.
\end{gather*}
\end{cor}

\begin{proof}
Since $C_2^2=q^{-2}$, Theorem \ref{thm:comod} does not literally apply. Nevertheless,
in the proof there, the only place where this matters is the term $B_0$ in degree $0$.
Since this term is identically zero in $\mathbb{F}^{\ssD}$, the same proof works.
\end{proof}

\begin{cor}
Let $V$ be an admissible $\E_n$ module, where $C$ acts as scalar $C_1$ with $q C_1^2\neq \pm1,q^{\pm2}$.
Then $V\otimes \mathbb{F}^{\ssB}$ is an admissible $\cK_n$ module with the action
\begin{gather*}
\Delta E_i(z)= e_i\big(q^{1/2}z\big)\otimes K_i^-(z)
+1\otimes k\tilde{K}^-_i(z) \tilde{K}^+_i(qz)
+\ft_i\big(q^{-1/2}z\big)\otimes K^+_i(z),
\\
\Delta K_i^+(z)= \psi_i^+\big(q^{-1/2}C_1^{-1}z\big)^{-1}\otimes K_i^+(z),
\qquad
\Delta K_i^-(z)=\psi_i^-\big(q^{1/2}z\big)\otimes K_i^-(z),
\\
\Delta C=C_1q^{1/2},
\end{gather*}
where $\tilde{K}^\pm_i(z)$ and $k$ are as in Lemma~$\ref{lem:typeB}$.
\end{cor}

\begin{Example}
Consider the $\cK_n$ module
\begin{gather*}
\mathbb{F}(u_1)\otimes\cdots\otimes\mathbb{F}(u_\ell)\otimes \mathbb{F}^{\ssD},
\qquad \ell\ge3.
\end{gather*}
The central element acts as $C^2=q^{2\ell-2}$.
The current $E_i(z)$ is represented as a sum of $2\ell$ vertex operators:
\begin{gather}
z^{-1}E_i(z)=
-\sum_{k=1}^\ell q^{k-\ell}u_k\Lambda_{i, k}(z)
+\sum_{k=1}^\ell q^{-k+\ell-1}u_k^{-1}\Lambda_{i,\bar k}(z),
\label{sumLambda}
\end{gather}
where
\begin{gather*}
\Lambda_{i,k}(z)=\overbrace{1\otimes \cdots\otimes 1}^{k-1}\otimes
{:}V_i\big(q^{k-\ell+1}z\big)^{-1}{:}\otimes \psi_i^+\big(q^{k-\ell+1}z\big) ^{-1}\otimes
\cdots \otimes \psi^+_i(z)^{-1}\otimes K_i^+(z),
\\
\Lambda_{i,\bar k}(z)=\overbrace{1\otimes \cdots\otimes 1}^{k-1}\otimes
V_i\big(q^{-k+\ell-1}z\big)\otimes\psi_i^-\big(q^{-k+\ell-2}z\big)\otimes
\cdots \otimes \psi^-_i\big(q^{-1}z\big)\otimes K_i^-(z).
\end{gather*}
In the right-hand side,
$\psi^\pm_i(z)$ or $K^\pm_i(z)$ stand for their action on $\mathbb{F}(u_i)$ or $\mathbb{F}^{\ssD}$, respectively.

In what follows we introduce an ordering $1\!\prec \cdots \!\prec \ell\!\prec \bar\ell\!\prec \cdots\!\prec \bar1$
to the set $\{1,\dots,\ell,\bar\ell,\dots,\bar 1\}$. We use letters $a,b,\dots$
for elements of the latter, with the convention that $\bar{\bar a}=a$.
Each $\Lambda_{i,a}(z)$ is a product of oscillator part $\Lambda^{\rm osc}_{i,a}(z)$ and the zero mode part.
The contractions of the former are given by the following table.

\begin{table}[H]\renewcommand{\arraystretch}{1.7}\setlength{\tabcolsep}{2.5pt}
\centering
\caption{Contractions $\cont{\Lambda^{\rm osc}_{i,a}(z)}{\Lambda^{\rm osc}_{j,b}(w)}$ with $x=w/z$}
{\small\begin{tabular}{c|c|c|c|c|c}
\hline
& $\begin{array}{c}a\prec b\\[-2ex] b\neq \bar a\end{array}$
& $a=b$
&\!\!$\begin{array}{c}a\succ b \\[-2ex] b\neq \bar a\end{array}$\!\!
& $\begin{array}{c}a\prec b\\[-2ex] b=\bar a\end{array}$
& $\begin{array}{c}a\succ b\\[-2ex] b=\bar a\end{array}$
\\
\hline
$i\equiv j$ & $\dfrac{1\!-\!q_2^{-1}x}{1-q_2x}$ & $(1\!-\!x)(1\!-\!q_2^{-1}x)$ & $1$ &
\!\!$\begin{array}{c}\dfrac{1-q_2^{-1}x}{1-q_2x}\times
\\[1ex]
\dfrac{1}{(1!-\!C^2q_2^{1-a}x)(1\!-\!C^2q_2^{2-a}x)}\end{array}$\!\!
& $\dfrac{1}{(1\!-\!C^{-2}q_2^{b}x)(1\!-\!C^{-2}q_2^{b-1}x)}$
\\
\hline
$i\!+\!1\!\equiv\! j$ & ${\dfrac{1\!-\!q_3^{-1}x}{1-q_1x}}$ & $\dfrac{1}{1-q_1x}$ & $1$ &
$\dfrac{1-q_3^{-1}x}{1-q_1x}$ $(1-C^2q_2^{1-a}q_1x)$
& $1-C^{-2}q_2^{b-1}q_3^{-1}x$
\\
\hline
\end{tabular}}
\end{table}

The contractions in the case $i-1\equiv j$ can be obtained by switching $q_1$ and $q_3$.

Using this table one can check directly the wheel conditions, namely that $E_{i,i,i+1}(z_1,z_2,w)$ represented by
\begin{gather*}
\frac{1-C^2z_2/z_1}{1-z_2/z_1}\frac{1-C^{-2}z_2/z_1}{1-z_2/z_1}\frac{1-q_2z_2/z_1}{1-z_2/z_1}
\frac{1-q_1w/z_1}{1-w/z_1}\frac{1-q_1w/z_2}{1-w/z_2}
\, \Lambda_{i,a}(z_1)\Lambda_{i,b}(z_2)\Lambda_{i+1,c}(w)
\end{gather*}
vanishes for all $a,b,c\in\{1,\dots,\ell,\overline{\ell},\dots,\bar 1\}$ under the following specializations:
\begin{enumerate}\itemsep=0pt
\item[$(i)$] $(z_1,z_2,w)=\big(z,q_{2}z,q_1^{-1}z\big)$,
\item[$(ii)$] $(z_1,z_2,w)=\big(z,C^{2}z,C^{2}q_3z\big)$,
\item[$(iii)$] $(z_1,z_2,w)=\big(z,C^{-2}z,C^{-2}q_1^{-1}z\big)$.
\end{enumerate}

\medskip
\noindent{\bf Case $\boldsymbol{(i)}$.} This follows from
\begin{gather*}
\cont{\Lambda_{i,a}(z_1)}{\Lambda_{i,b}(z_2)}\bigl|_{z_2=q_2 z_1}=0
\qquad \text{if}\quad a\preceq b,
\\
\cont{\Lambda_{i,b}(z_2)}{\Lambda_{i+1,c}(w)}\bigl|_{w=q_3 z_2}=0
\qquad \text{if}\quad b\prec c,
\\
(1-q_1w/z_1)\cont{\Lambda_{i,a}(z_1)}{\Lambda_{i+1,c}(w)}\bigl|_{w=q_1^{-1} z_1}=0
\qquad \text{if}\quad a\succ c.
\end{gather*}

\medskip
\noindent{\bf Case $\boldsymbol{(ii)}$.} This follows from
\begin{gather*}
\big(1-C^{-2}z_2/z_1\big)\cont{\Lambda_{i,a}(z_1)}{\Lambda_{i,b}(z_2)}\bigl|_{z_2=C^{2}z_1}=0
\qquad \text{if}\quad (a,b)\neq (\bar 1,1),
\\
\cont{\Lambda_{i,1}(z_2)}{\Lambda_{i+1,c}(w)}\bigl|_{w=q_3z_2}=0
\qquad \text{if}\quad c\neq 1,
\\
\cont{\Lambda_{i,\bar 1}(z_1)}{\Lambda_{i+1, 1}(w)}\bigl|_{w=C^{2}q_3z_1}=0.
\end{gather*}

\medskip
\noindent{\bf Case $\boldsymbol{(iii)}$.}
This follows from
\begin{gather*}
\big(1-C^2z_2/z_1\big)\cont{\Lambda_{i,a}(z_1)}{\Lambda_{i,b}(z_2)}\bigl|_{z_2=C^{-2}z_1}=0
\qquad \text{if}\quad (a,b)\neq (1,\bar 1),
\\
(1-q_1w/z_2)\cont{\Lambda_{i,\bar 1}(z_2)}{\Lambda_{i+1,c}(w)}\bigl|_{w=q_1^{-1}z_2}=0
\qquad \text{if}\quad c\neq \bar 1,
\\
\cont{\Lambda_{i,1}(z_1)}{\Lambda_{i+1,\bar 1}(w)}\bigl|_{w=C^{-2}q_1^{-1}z_1}=0.
\end{gather*}
Similarly for $E_{i,i,i-1}(z_1,z_2,w)$.
\end{Example}

\begin{rem}
As it turns out, the currents \eqref{sumLambda} commute with
a system of screening operators derived from quantum toroidal
algebra of type $D$. We hope to address this point elsewhere.
\end{rem}

\section{Integrals of motion}\label{sec:IM}
We continue to work with admissible representations of $\cK_n$ on which $C$ acts as a scalar.
Our goal in this section is to introduce a family of commuting operators acting on each such representation
(with mild restrictions on parameters),
which we call integrals of motion (IMs). Loosely speaking, we may think of them as generators of
a commutative subalgebra of $\cK_n$.

The IMs will depend on $q$, $d$, $C$ as well as additional arbitrary parameters
\begin{gather}
\mu_1,\dots,\mu_n\in\C,\qquad \sum_{i=1}^n\mu_i=0.\label{mui}
\end{gather}
In addition, they depend also on an ``elliptic'' parameter $p$. However,
our experience with the algebra $\cK_1$
tells that the IMs exist only when $p$ and $C$ are related in a special way.
Guided by the result in \cite{FJMV}, we shall impose the relation
\begin{gather*}
C^2=pq^2
\end{gather*}
and make the following assumption:
\begin{gather}
|p|<1,\qquad
\big|qd^{\pm1}\big|<1,\qquad \big|pq^{-2}\big|<1,\qquad \big|pq^{-1}d^{\pm1}\big|<1.
\label{assume-para}
\end{gather}
In particular $|q^2|<1$. We use parameters $\tau\in\C $ with $\Im \tau>0$ and $\beta,\gamma\in\C$ given by
\begin{gather}
p={\rm e}^{-2\pi{\rm i} /\tau},\qquad q^2=p^\gamma,\qquad d=p^\beta.
\label{gamma-beta}
\end{gather}

We use also $\theta(u)=p^{u^2/2-u/2}\Theta_p(p^u)$, which satisfies
\begin{gather}
\theta(u+1)=-\theta(u),\qquad \theta(u+\tau)=-{\rm e}^{-2\pi {\rm i} u-\pi {\rm i}\tau}\theta(u).
\label{theta-per0}
\end{gather}

\subsection{Dressed currents}

A basic constituent of IMs is the dressed currents, obtained by modifying the currents $E_i(z)$ of~$\cK_n$.

In the following we 
define operators $H_{i,0}$ by
$(K_{i,0}^-)^{-1}K_{i,0}^+=q^{2H_{i,0}}$.
The dressed currents of~$\cK_n$ are defined by
\begin{gather*}
\bE_i(z)=\bK_i^-(z)^{-1}E_i(z) z^{-\gamma H_{i,0}-\gamma+\mu_i},\\
\bK_i^-(z)=\exp\bigg({-}\big(q-q^{-1}\big)\sum_{r>0}\frac{H_{i,-r}}{1-p^r}z^r\bigg),
\end{gather*}
where $\gamma$ and $\mu_i$ are as in \eqref{gamma-beta} and \eqref{mui}.

The dressed currents have well-defined matrix coefficients on admissible representations.
They satisfy the exchange relations as meromorphic functions
\begin{gather*}
\lambda_{i,j}(z_2/z_1)\bE_i(z_1)\bE_j(z_2)=\lambda_{j,i}(z_1/z_2)\bE_j(z_2)\bE_i(z_1),
\end{gather*}
with
\begin{gather*}
\lambda_{i,i}(x)=x^{-\gamma}
\frac{1-C^2x}{1-x}\frac{1-C^{-2}x}{1-x}\frac{1-q^2x}{1-x}
\frac{\big(q^{-2}x;p\big)_\infty}{(q^{2}x;p)_\infty},
\\
\lambda_{i,i\pm1}(x)=x^{\gamma/2}
d^{\mp1/2} \frac{1-q^{-1}d^{\pm1}x}{1-x}
\frac{\big(q d^{\pm1}x;p\big)_\infty}{\big(q^{-1}d^{\pm1}x;p\big)_\infty},
\end{gather*}
and $\lambda_{i,j}(x)=1$ in all other cases.
Up to a power of $x$, $\lambda_{i,j}(x)$ specializes to $\lambda^0_{i,j}(x)$ at $p=0$.
Each matrix coefficient of products of dressed currents is a meromorphic function (up to an overall power function).
We have ``elliptic'' exchange relations as meromorphic functions
\begin{gather*}
\bE_{i}(z)\bE_i(w)=\bE_i(w)\bE_{i}(z)\frac{\theta(v-u+\gamma)}{\theta(v-u-\gamma)},
\\
\bE_{i}(z)\bE_{i\pm1}(w)=\bE_{i\pm1}(w)\bE_{i}(z)
q^{\pm 2\beta} \frac{\theta(v-u\pm \beta-\gamma/2)}{\theta(v-u\pm \beta+\gamma/2)},
\end{gather*}
where $z=p^u$, $w=p^v$.

As before we define
\begin{gather}
\bE_{i_1,\dots,i_N}(z_1,\dots, z_N)=\prod_{1\le r<s\le N} \lambda_{i_r,i_s}(z_s/z_r)
 \bE_{i_1}(z_1)\cdots \bE_{i_N}(z_N).
\label{btEiii}
\end{gather}
Then it has the same symmetry as \eqref{E-symm}, and we have
\begin{gather*}
\bE_{i_1,\dots,i_N}(z_1,\dots, z_N)=\prod_{r<s}z_s^{-\gamma(\bar\alpha_{i_r},\bar\alpha_{i_s})}
\prod_{s=1}^N \bK_{i_s}^-(z_{s})^{-1} E_{i_1,\dots,i_N}(z_1,\dots, z_N) \prod_{s=1}^N z_s^{-\gamma H_{i_s,0}-\gamma+\mu_{i_s}}.
\end{gather*}
The product $\prod_{s=1}^N \bK_{i_s}^-(z_{s})^{-1}$ contains only $\{H_{i,-r}\}$ with $r>0$.
Its contribution to any matrix coefficient of $\bE_{i_1,\dots,i_N}(z_1,\dots, z_N)$
is a Laurent polynomial in $z_s$'s for degree reasons.
Hence the equality above shows that, up to an overall power of $z_i$'s,
each matrix coefficient of operator~\eqref{btEiii}
has the same functional form \eqref{Laurent} as that of
$E_{i_1,\dots,i_N}(z_1,\dots, z_N)$. In particular the wheel conditions \eqref{wheel1}--\eqref{wheel3} are valid also for
\eqref{btEiii}.

\subsection{Integrals of motion}\label{subsec:IM}

{\sloppy
Given parameters \eqref{mui},
let ${X}_n={X}_n(\mu_1,\dots,\mu_n)$ denote the space of entire functions $\vartheta(u_1,\dots,u_n)$ satisfying
\begin{gather}
\vartheta(u_1,\dots,u_i+1,\dots,u_n)=\vartheta(u_1,\dots,u_i,\dots,u_n),
\label{theta-per1}
\\
\vartheta(u_1,\dots,u_i+\tau,\dots,u_n)={\rm e}^{-2\pi {\rm i}(2u_i-u_{i-1}-u_{i+1}-\mu_i+\tau)}\,
\vartheta(u_1,\dots,u_i,\dots,u_n), 
\label{theta-per2}
\end{gather}
where $u_{i+n}=u_i$.
We have $\dim X_n=n$.}

For $M\ge1$ and $\vartheta\in X_n$, we introduce the kernel functions
\begin{gather}
h_{M}(\bu_1,\dots, \bu_n;\vartheta)\nn
\\ \qquad
{}=\frac{\vartheta(\ovu_1,\dots,\ovu_n)
{\prod_{i=1}^{n}\prod_{1\le a<b\le M}\theta(u_{i,a}-u_{i,b})\theta(u_{i,a}-u_{i,b}-\gamma)}}
{{\prod_{i=1}^{n-1}\prod_{1\le a,b\le M}\theta\big(u_{i,a}-u_{i+1,b}-\beta-\frac{\gamma}{2}\big)
\prod_{1\le a,b\le M}\theta\big(u_{n,a}-u_{1,b}-\beta+\frac{\gamma}{2}\big)}},
\label{kernel_h}
\end{gather}
where
\begin{gather*}
\bu_i=(u_{i,1},\dots,u_{i,M}),\qquad \ovu_i=\sum_{a=1}^M u_{i,a}.
\end{gather*}
We define the integrals of motion $\bG_M(\vartheta)$ of $\cK_n$ by
\begin{gather}
\bG_M(\vartheta)=\int\!\!\cdots\!\!\int
\prod_{1\le i\le n}^{\curvearrowleft}\prod_{1\le a\le M}^{\curvearrowleft}\bE_{i}(x_{i,a})
\cdot
h_{M}(\bu_1,\dots, \bu_n;\vartheta)\ \prod_{i=1}^n\prod_{a=1}^M \frac{{\rm d}x_{i,a}}{x_{i,a}}.
\label{IM1}
\end{gather}
Here $x_{i,a}=p^{u_{i,a}}$, and
the symbol $\prod_{1\le i\le N}^{\curvearrowleft}$ stands for ordered product
\[
\prod_{1\le i\le N}^{\curvearrowleft}A_i=A_NA_{N-1}\cdots A_1.
\]
Even though individual factors contain fractional powers,
the integrand of \eqref{IM1} comprises only integer powers with respect to each $x_{i,a}$.
The integrals are taken over a common circle $|x_{i,a}|=R$. The result is independent of the choice of $R>0$.

Formally, \eqref{IM1} looks identical to that of the integrals of motion of type A \cite{FJM2,FKSW,KS}.
In~particular the kernel functions \eqref{kernel_h} are the same as those used there. The difference is hidden in the pole structure of
products of dressed currents.

For later use we rewrite the integrand in terms of the current in \eqref{btEiii},
\begin{gather*}
\bE^{(M)}(\bx_1,\dots,\bx_n)=\bE_{1,\dots,1,\dots,n,\dots,n}(\bx_1,\dots,\bx_n),
\end{gather*}
where $\bx_i=(x_{i,1},\dots,x_{i,M})$.
In the right-hand side, each index $i$ 
occurs $M$ times.

\begin{prop}
Set $x=p^u$ and
\begin{gather*}
\xi(u)=p^{u^2/2}(1-x)\big(x,p^2q^2x,p;p\big)_\infty,
\\
\eta(u)=p^{u^2/2-(\beta+1/2)u}\frac{d^{-1/2}}{1-x^{-1}}\big(qd^{-1}x,qdx^{-1},p;p\big)_\infty.
\end{gather*}
Then the IM's can be expressed as
\begin{gather}
\bG_M(\vartheta)=c_M \int\!\!\cdots\!\!\int
\bE^{(M)}(\bx_1,\dots,\bx_n)\,k_M(\bu_1,\dots,\bu_n;\vartheta)
\prod_{i=1}^n\prod_{a=1}^M\frac{{\rm d} x_{i,a}}{x_{i,a}},
\label{IMEM}
\end{gather}
where $c_M$ is a constant, and
\begin{gather*}
k_M(\bu_1,\dots,\bu_n;\vartheta)=
\vartheta(\ovu_1,\dots,\ovu_n)\,
\frac{\prod_{i=1}^n\prod_{1\le a\neq b\le M}
\xi(u_{i,a}-u_{i,b})}
{\prod_{i=1}^n\prod_{1\le a,b\le M}
\eta(u_{i,a}-u_{i+1,b})},
\\
x_{i,a}=p^{u_{i,a}},\qquad \ovu_i=\sum_{a=1}^M u_{i,a}.
\end{gather*}
Here and after, we set $u_{n+1,a}=u_{1,a}$.
\end{prop}

\begin{proof}
This follows from the relations
\begin{gather*}
\frac{\lambda_{i,i}(p^{-u})}{\lambda_{i,i}(p^u)}=\frac{\theta(u+\gamma)}{\theta(u-\gamma)},\qquad
\frac{\lambda_{i,i\mp1}(p^{-u})}{\lambda_{i,i\pm1}(p^u)}=
q^{\pm 2\beta} \frac{\theta(u\pm \beta-\gamma/2)}{\theta(u\pm \beta+\gamma/2)},
\\
\lambda_{i,i}(p^u)\xi(u)\xi(-u)=-p^{-\gamma^2/2-3\gamma/2-1}\theta(u)\theta(u-\gamma),
\\
\lambda_{i,i\pm 1}\big(p^{\mp u}\big)\theta(u\pm \gamma/2-\beta)=p^{(\gamma/2\mp \beta)^2/2-(\gamma/2-\beta)/2}\, \eta(u).\tag*{\qed}
\end{gather*}
\renewcommand{\qed}{}
\end{proof}

\subsection{Commutativity}
We are now in a position to state the commutativity of IMs.

\begin{thm}\label{thm:main}
For all $M,N\ge1$ and $\vartheta_1,\vartheta_2\in X_n$, the integrals of motion mutually commute:
\begin{gather*}
[\bG_M(\vartheta_1),\bG_N(\vartheta_2)]=0.
\end{gather*}
\end{thm}

Theorem \ref{thm:main} is the main result of this paper.
The rest of this section is devoted to its proof.

\begin{proof}
We start by examining the product of operators, dropping irrelevant constants
\begin{gather*}
\bG_M(\vartheta_1)=
\int\!\!\cdots\!\!\int_{\cC_1}
\bE^{(M)}(\bx_1,\dots,\bx_n)\,k_M(\bu_1,\dots,\bu_n;\vartheta_1)
\prod_{i=1}^n\prod_{a=1}^M\frac{{\rm d} x_{i,a}}{x_{i,a}},
\\
\bG_N(\vartheta_2)=
\int\!\!\cdots\!\!\int_{\cC_R}
\bE^{(N)}(\by_1,\dots,\by_n)\,k_N(\bv_1,\dots,\bv_n;\vartheta_2)
\prod_{i=1}^n\prod_{b=1}^N\frac{{\rm d} y_{i,b}}{y_{i,b}}.
\end{gather*}
Here $x_{i,a}=p^{u_{i,a}}$, $y_{i,b}=p^{v_{i,b}}$.
We choose the contours to be the unit circle $\cC_1\colon |x_{i,a}|=1$ for~$\bG_M(\vartheta_1)$, and
a circle $\cC_R\colon |y_{i,b}|=R$ of radius $R$ for $\bG_N(\vartheta_2)$.

Consider first the product
\begin{align}
\bG_N(\vartheta_2)\bG_M(\vartheta_1)
=\!\!\!{}&\int\!\!\cdots\!\!\int_{\cC_R}\int\!\!\cdots\!\!\int_{\cC_1}\!\!
\bE^{(N)}(\by_1,\dots,\by_n)\,
\bE^{(M)}(\bx_1,\dots,\bx_n)
\nn
\\
&\!\!\times k_N(\bv_1,\dots,\bv_n;\vartheta_2)\,k_M(\bu_1,\dots,\bu_n;\vartheta_1)
\prod_{i=1}^n\prod_{b=1}^N\frac{{\rm d} y_{i,b}}{y_{i,b}}\prod_{i=1}^n\prod_{a=1}^M\frac{{\rm d} x_{i,a}}{x_{i,a}}.
\label{G2G1}
\end{align}
We choose $R\gg 1$ to ensure that the product of operators in the integrand converges absolutely.
Introducing the combined variables $\bz_i=(z_{i,1},\dots,z_{i,M+N})$, where
$z_{i,a}=x_{i,a}$, $1\le a\le M$, $z_{i,a+M}=y_{i,a}$, and $u_{i,M+a}=v_{i,a}$, $1\le a\le N$,
and substituting
\begin{gather*}
\bE^{(N)}(\by_1,\dots,\by_n)\bE^{(M)}(\bx_1,\dots,\bx_n)
\\ \qquad
{}=\bE^{(M+N)}(\bz_1,\dots,\bz_n)
\prod_{i=1}^n\prod_{\substack{1\le a\le M\\ 1\le b\le N}}\lambda_{i,i}(x_{i,a}/y_{i,b})^{-1}
\prod_{i=1}^{n}\prod_{\substack{1\le a\le M\\ 1\le b\le N}}\lambda_{i+1,i}(x_{i,a}/y_{i+1,b})^{-1}
\\ \qquad \hphantom{=}
{}\times\prod_{i=1}^{n}\prod_{\substack{1\le a\le M\\ 1\le b\le N}}\lambda_{i-1,i}(x_{i,a}/y_{i-1,b})^{-1},
\end{gather*}
we find that the integrand of \eqref{G2G1} can be written as a product of two factors $ST$, with
\begin{gather}
S(\bz_1,\dots,\bz_n)\nn
\\ \qquad
{}=\bE^{(M+N)}(\bz_1,\dots,\bz_n)
\prod_{i=1}^n\prod_{1\le a\neq b\le M+N}\xi(u_{i,a}-u_{i,b})\,
\prod_{i=1}^n\prod_{1\le a, b\le M+N}\eta(u_{i,a}-u_{i+1,b})^{-1},\nn
\\
T(\bu_1,\dots,\bu_n|\bv_1,\dots,\bv_n)\nn
\\ \qquad
{}=\vartheta_1(\ovu_1,\dots,\ovu_n)\vartheta_2(\ovv_1,\dots,\ovv_n)\nn
\\ \qquad\hphantom{=}
{}\times\prod_{i=1}^{n}\prod_{\substack{1\le a\le M\\ 1\le b\le N}}
\frac{\theta(u_{i,a}-v_{i+1,b}-\gamma/2-\beta)\theta(v_{i,b}-u_{i+1,a}+\gamma/2-\beta)}
{\theta(u_{i,a}-v_{i,b})\theta(u_{i,a}-v_{i,b}-\gamma)}.
\label{G2G1-int2}
\end{gather}

Starting from $R\gg1$, we bring the contours $\cC_R$ for $y_{i,b}$ to the unit circle.
For that purpose we need to locate the position of poles between the two circles $\cC_R$ and $\cC_1$.
The poles between $y$'s and $x$'s come from $\xi$, $\eta$ and $\theta$:
\begin{gather*}
\prod_{i=1}^n\prod_{\substack{1\le a\le M\\ 1\le b\le N}}
(1-x_{i,a}/y_{i,b})(1-y_{i,b}/x_{i,a})
\big(\cancel{x_{i,a}/y_{i,b}},p^2q^2 x_{i,a}/y_{i,b},
\overset{\scriptstyle 1-y_{i,b}/x_{i,a}}{\cancel{y_{i,b}/x_{i,a}}},p^2q^2 y_{i,b}/x_{i,a};p\big)_\infty
\\ \qquad
{}\times
\prod_{i=1}^n\prod_{\substack{1\le a\le M\\ 1\le b\le N}}
\frac{(1-y_{i+1,b}/x_{i,a})(1-x_{i+1,a}/y_{i,b})}
{\big(qd^{-1}x_{i,a}/y_{i+1,b},\cancel{qd y_{i+1,b}/x_{i,a}},\cancel{qd^{-1}y_{i,b}/x_{i+1,a}},
qd x_{i+1,a}/y_{i,b};p\big)_\infty}
\\ \qquad
{}\times
\prod_{i=1}^n\prod_{\substack{1\le a\le M\\ 1\le b\le N}}
\frac{\big(pq^{-1}d^{-1}x_{i,a}/y_{i+1,b},
\cancel{qd y_{i+1,b}/x_{i,a}},\cancel{q d^{-1}y_{i,b}/x_{i+1,a}}, pq^{-1}dx_{i+1,a}/y_{i,b};p\big)_\infty}
{\big(\cancel{x_{i,a}/y_{i,b}},\cancel{py_{i,b}/x_{i,a}},q^{-2}x_{i,a}/y_{i,b},pq^2y_{i,b}/x_{i,a};p\big)_\infty}.
\end{gather*}
Under our assumption on parameters \eqref{assume-para}, we see that
the only poles between $\cC_R$ and $\cC_1$ are
\begin{gather}
y_{i,b}=q_2^{-1}x_{i,a}, p^{-1}q_2^{-1}x_{i,a},
\label{pick-poles}
\qquad 1\le i\le n,\quad 1\le a\le M,\quad 1\le b\le N.
\end{gather}
{\samepage Including the poles between $y$'s which come from $\eta$, poles of the integrand with respect to each~$y_{i,b}$ are as follows
(see the figure below):
\begin{gather*}
y_{i,b}=p^{-k}q_1 y_{i+1,c},\qquad p^{-k}q_3 y_{i-1,c},\qquad k\ge 0,
\\
y_{i,b}=p^{k}q_3^{-1} y_{i+1,c},\qquad p^{k}q_1^{-1} y_{i-1,c},\qquad k\ge 0,
\\
y_{i,b}=p^{k}q_2^{-1}x_{i,a},\qquad k\ge -1,
\\
y_{i,b}=p^k q_1^{-1}x_{i-1,a},\qquad p^kq_3^{-1}x_{i+1,a},\qquad k\ge 0.
\end{gather*}

}

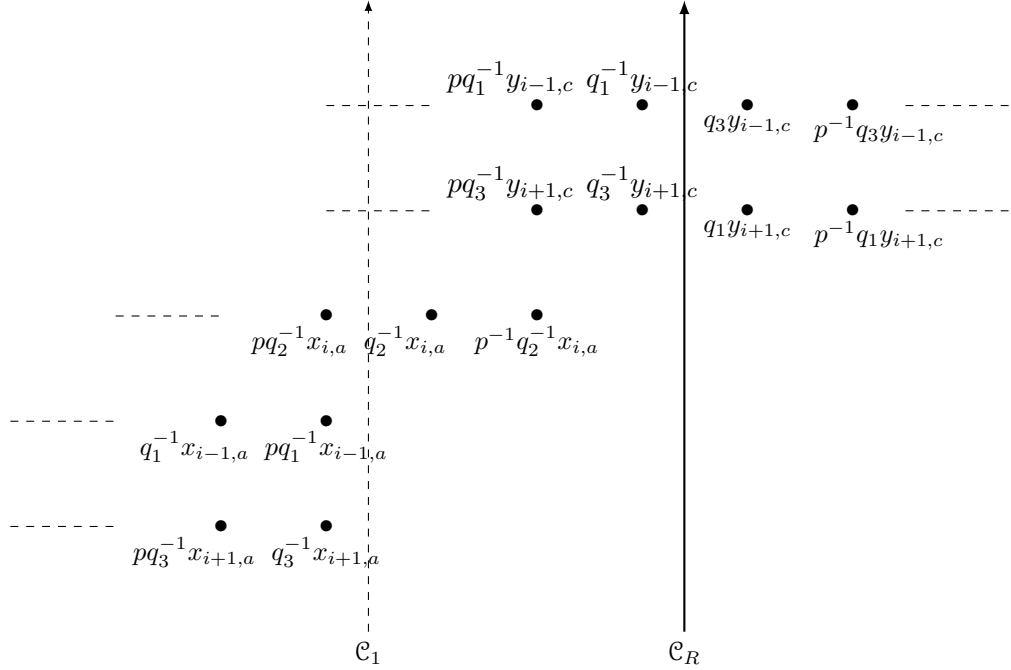
\begin{figure}[H]
\begin{center}
\begin{tikzpicture}[scale=0.7]

\draw[dashed] (-9,5)--(-7,5);
\node at (-5,5){$\bullet$};
\node at (-3,5){$\bullet$};
\node [above] at (-5.5,5){$pq_1^{-1}y_{i-1,c}$};
\node [above] at (-3,5){$q_1^{-1}y_{i-1,c}$};

\node at (-1,5){$\bullet$};
\node at (1,5){$\bullet$};
\draw[dashed] (2,5)--(4,5);
\node [below] at (-1,5) {\small $q_3 y_{i-1,c}$};
\node [below] at (1.5,5) {\small $p^{-1}q_3 y_{i-1,c}$};

\draw[dashed] (-9,3)--(-7,3);
\node at (-5,3){$\bullet$};
\node at (-3,3){$\bullet$};
\node [above] at (-5.5,3){$pq_3^{-1}y_{i+1,c}$};
\node [above] at (-3,3){$q_3^{-1}y_{i+1,c}$};

\node at (-1,3){$\bullet$};
\node at (1,3){$\bullet$};
\draw[dashed] (2,3)--(4,3);
\node [below] at (-1,3) {\small $q_1y_{i+1,c}$};
\node [below] at (1.5,3) {\small $p^{-1}q_1y_{i+1,c}$};

\draw[dashed] (-13,1)--(-11,1);
\node at (-9,1){$\bullet$};
\node at (-7,1){$\bullet$};
\node at (-5,1){$\bullet$};
\node [below] at (-9.5,1) {\small $pq_2^{-1}x_{i,a}$};
\node [below] at (-7.5,1) {\small $q_2^{-1}x_{i,a}$};
\node [below] at (-5,1) {\small $p^{-1}q_2^{-1}x_{i,a}$};

\draw[dashed] (-15,-1)--(-13,-1);
\node at (-11,-1){$\bullet$};
\node at (-9,-1){$\bullet$};
\node [below] at (-11.5,-1) {\small $q_1^{-1}x_{i-1,a}$};
\node [below] at (-9,-1) {\small $pq_1^{-1}x_{i-1,a}$};

\draw[dashed] (-15,-3)--(-13,-3);
\node at (-11,-3){$\bullet$};
\node at (-9,-3){$\bullet$};
\node [below] at (-11.5,-3) {\small $pq_3^{-1}x_{i+1,a}$};
\node [below] at (-9,-3) {\small $q_3^{-1}x_{i+1,a}$};

\node [below] at (-2.2,-5) {\small $\mathcal{C}_R$}; 
\draw [-latex,thick] (-2.2,-5)--(-2.2,7);

\node [below] at (-8.2,-5) {\small $\mathcal{C}_1$}; 
\draw [-latex,dashed] (-8.2,-5)--(-8.2,7);
\end{tikzpicture}
\caption{The contour in the $y_{i,b}$ plane. \label{Fig0}}
\end{center}
\end{figure}

We move the contour to the unit circle, picking up residues of poles \eqref{pick-poles} along the way.
At~first glance, taking residues in one variable $y_{i,b}=p^{-1}q_2^{-1}x_{i,a}$ or $y_{i,b}=q_2^{-1}x_{i,a}$ seems to produce new poles
in other variables,
\begin{gather*}
y_{i+1,c}=q_2^{-1}q_1^{-1}x_{i,a},\quad p^{-1}q_2^{-1}q_1^{-1}x_{i,a},
\\
y_{i-1,c}=q_2^{-1}q_3^{-1}x_{i,a},\quad p^{-1}q_2^{-1}q_3^{-1}x_{i,a}.
\end{gather*}
An important point is that these poles are cancelled by the zeros of $\bE^{(M+N)}(\bz_1,\dots,\bz_n)$
due to the wheel conditions \eqref{wheel1}--\eqref{wheel3}.

For the product in the opposite order
\begin{align}
\bG_M(\vartheta_1)\bG_N(\vartheta_2)
=\!\!\!{}&\int\!\!\cdots\!\!\int_{\cC_1} \int\!\!\cdots\!\!\int_{\cC_R}
\bE^{(M)}(\bx_1,\dots,\bx_n)\,
\bE^{(N)}(\by_1,\dots,\by_n)\nn
\\
&\times k_M(\bu_1,\dots,\bu_n;\vartheta_1)k_N(\bv_1,\dots,\bv_n;\vartheta_2)
\prod_{i=1}^n\prod_{a=1}^M\frac{{\rm d} x_{i,a}}{x_{i,a}}\prod_{i=1}^n\prod_{b=1}^N\frac{{\rm d} y_{i,b}}{y_{i,b}},\!
\label{G1G2}
\end{align}
we choose $R\ll 1$ for convergence.
The integrand can be obtained from the one for \eqref{G2G1}
by switching the roles of $\bx$, $M$ and $\by$, $N$.
It amounts to changing \eqref{G2G1-int2} to
{\samepage\begin{align*}
T'(\bu_1,\dots,\bu_n|\bv_1,\dots,\bv_n)={}&\vartheta_1(\ovu_1,\dots,\ovu_n)\vartheta_2(\ovv_1,\dots,\ovv_n)
\\
& \times\prod_{i=1}^{n}\prod_{\substack{1\le a\le M\\ 1\le b\le N}}
\frac{\theta(v_{i,b}\!-\!u_{i+1,a}\!-\!\gamma/2\!-\!\beta)\theta(u_{i,a}\!-\!v_{i+1,b}\!+\!\gamma/2\!-\!\beta)}
{\theta(v_{i,b}-u_{i,a})\theta(v_{i,b}-u_{i,a}-\gamma)}.
\end{align*}}\noindent
Similarly as above, we start from $R\ll1$ and bring $\cC_R$ to the unit circle.
The relevant poles in between are
\begin{gather*}
y_{i,b}=q_2 x_{i,a},\, pq_2 x_{i,a},
\qquad 1\le i\le n,\quad 1\le a\le M, \quad 1\le b\le N.
\end{gather*}
Moving the contour we pick up residues of these poles.
Again the wheel conditions ensure that no new poles arise for the remaining variables.

In order to prove the commutativity, we must show that, after moving contours to the unit circle,
\eqref{G2G1} and \eqref{G1G2} give the same result.

First, let us compare terms that arise without taking residues.
Since all integrals are taken on the unit circle, we can symmetrize the integrand.
Taking into account the symmetry of $S(\bz_1,\dots,\bz_n)$, we see that
the equality of integrals reduces to the identity of theta functions
\begin{gather}
\Sym\,T(\bu_1,\dots,\bu_n|\bv_1,\dots,\bv_n)
=\Sym\, T'(\bu_1,\dots,\bu_n|\bv_1,\dots,\bv_n),
\label{KSid}
\end{gather}
where $\Sym$ stands for symmetrization in each group of variables
$\{u_{i,a}\}_{a=1}^M\cup\{v_{i,b}\}_{b=1}^N$, $1\le i\le n$.
This identity was stated in \cite{KS}.
Since their proof contains some gaps, we prove it in Appendix; see Theorem \ref{thm:KSidentity}
and Remark \ref{rem:gap} below.

In general, we compare terms obtained by taking residues with respect to some group of variables.
First consider terms coming from \eqref{G2G1}.
In view of symmetry and the zeros of $S(\bz_1,\dots,\bz_n)$ at $y_{i,a}=y_{i,b}$, it is enough to consider the case
\begin{gather}
y_{i,a}=p^{-1}q_2^{-1}x_{i,a},\qquad 1\le a\le k_i,\label{group1}\\
y_{i,b}=q_2^{-1}x_{i,b},\qquad k_i+1\le b\le l_i,\label{group2}
\end{gather}
with some $0\le k_i\le l_i\le \min(M, N)$, $1\le i \le n$.
We compare it with the corresponding residue coming from \eqref{G1G2}
\begin{gather}
y_{i,a}=pq_2x_{i,a},\qquad 1\le a\le k_i,\label{group3}\\
y_{i,b}=q_2x_{i,b},\qquad k_i+1\le b\le l_i,\label{group4}
\end{gather}
for $1\le i \le n$. We further rename variables $x_{i,a}$ to $p^{-1}q^{-1}_2x_{i,a}$ for \eqref{group3} and
to $q^{-1}_2x_{i,a}$ for~\eqref{group4}.
The contours for them change to $\big|p^{-1}q^{-1}_2x_{i,a}\big|=1$ and $\big|q^{-1}_2x_{i,a}\big|=1$, respectively.
We shift them back to the unit circle $|x_{i,a}|=1$, noting that the wheel conditions ensure
there are no poles which hinder the shift.

The factor $S(\bz_1,\dots,\bz_n)$ has no poles at \eqref{group1} and \eqref{group2} and specializes to
\begin{gather}
S\big(\dots, z_{i,a},\dots,z_{i,b},\dots, p^{-1}q_2^{-1}z_{i,a},\dots,q_2^{-1}z_{i,b},\dots\big).
\label{Ssp1}
\end{gather}
Similarly, at \eqref{group3} and \eqref{group4} it specializes to
\begin{gather}
S(\dots, z_{i,a},\dots,z_{i,b},\dots, pq_2z_{i,a},\dots,q_2z_{i,b},\dots).
\label{Ssp2}
\end{gather}
After the renaming, \eqref{Ssp2} is brought to \eqref{Ssp1}
because $S(\bz_1,\dots,\bz_n)$ is symmetric in \linebreak $\{z_{i,a}\}_{1\le a\le M+N}$ for each~$i$.

It remains to compare the residues of the functions $T$, $T'$.
Since they are periodic with period~$1$, their residues at $v_{i,a}=u_{i,a}-\gamma-1$ are the same as
those at $v_{i,a}=u_{i,a}-\gamma$.
Hence we are to show the equality
\begin{gather*}
\mathrm{Sym}^*
\Res_{\substack{v_{i,a}=u_{i,a}-\gamma,\\1\le a\le l_i,\, 1\le i\le n}}
T(\bu_1,\dots,\bu_n|\bv_1,\dots,\bv_n)
\\ \qquad
{}=\mathrm{Sym}^*\Big\{(-1)^{\sum_{i=1}^n l_i}
\Res_{\substack{v_{i,a}=u_{i,a}+\gamma,\\ 1\le a\le l_i,\, 1\le i\le n}} T'(\bu_1,\dots,\bu_n|\bv_1,\dots,\bv_n)\Big\}
\Big|_{\substack{u_{i,a}\to u_{i,a}-\gamma,\\ 1\le a\le l_i,\, 1\le i\le n}},
\end{gather*}
where $\mathrm{Sym}^*$ stands for the symmetrization with respect to the
remaining variables $\{u_{i,a}\}_{a=l_i+1}^M\cup\{v_{i,b}\}_{b=l_i+1}^N$, $1\le i\le n$.
Note that moving the contours $\mathcal{C}_R$ from $R<1$ to $\mathcal{C}_1$ we obtain a sign factor.

To check the equality, we start from the residues of the identity \eqref{KSid}
\begin{gather}
\Res_{\substack{v_{i,a}=u_{i,a}-\gamma,\\ 1\le a\le l_i,\, 1\le i\le n}} \!\!\!\!\! \Sym\,T(\bu_1,\dots,\bu_n|\bv_1,\dots,\bv_n)
\!=\!\!\!\!\Res_{\substack{v_{i,a}=u_{i,a}-\gamma,\\ 1\le a\le l_i,\, 1\le i\le n}}\!\!\!\!\!\Sym\, T'(\bu_1,\dots,\bu_n|\bv_1,\dots,\bv_n).\!
\label{KSid-res}
\end{gather}
Symmetrization $\Sym$ amounts to replacing variables $\bu_i$, $\bv_i$ as
\begin{gather*}
\bu_i\to \{u_{i,a}\}_{a\in I_i}\cup\{v_{i,b}\}_{b\in J'_i},\qquad
\bv_i\to \{u_{i,a}\}_{a\in I'_i}\cup\{v_{i,b}\}_{b\in J_i},
\end{gather*}
and sum over partitions of indices
\begin{gather*}
I_i\sqcup I_i'=\{1,\dots,M\},\qquad J_i\sqcup J_i'=\{1,\dots,N\},\qquad
|I_i'|=|J_i'|.
\end{gather*}
In order to have non-zero residues at $v_{i,a}=u_{i,a}-\gamma$ ($1\le a\le l_i$)
in the left-hand side of \eqref{KSid-res},
$I_i$ and $J_i$ must contain $\{1,\dots,l_i\}$.
Therefore the left-hand side of \eqref{KSid-res} reduces to
\begin{gather*}
\mathrm{Sym}^*
\Res_{\substack{v_{i,a}=u_{i,a}-\gamma,\\ 1\le a\le l_i,\, 1\le i\le n}}
T(\bu_1,\dots,\bu_n|\bv_1,\dots,\bv_n).
\end{gather*}
Similarly, in the right-hand side of \eqref{KSid-res}, non-zero residues appear only when
$J_i'$ and $I_i'$ contain $\{1,\dots,l_i\}$.
For each $i$ we have
\begin{gather*}
\Res_{\substack{v_{i,a}=u_{i,a}-\gamma,\\ 1\le a\le l_i}} T'(\dots,v_{i,1},\dots,v_{i,l_i},\dots|\dots,u_{i,1},\dots,u_{i,l_i},\dots)
\\ \qquad
{}=(-1)^{l_i}\Res_{\substack{u'_{i,a}=v'_{i,a}+\gamma,\\ 1\le a\le l_i}}
T'(\dots,v'_{i,1},\dots,v'_{i,l_i},\dots|\dots,u'_{i,1},\dots,u'_{i,l_i},\dots)
\Big|_{\substack{v'_{i,a}\to u_{i,a}-\gamma\\ 1\le a\le l_i}}.
\end{gather*}
Altogether we can rewrite the right-hand side of \eqref{KSid-res} as
\begin{gather*}
\mathrm{Sym}^*\Bigl\{(-1)^{\sum_{i=1}^n l_i}
\Res_{\substack{v_{i,a}=u_{i,a}+\gamma,\\ 1\le a\le l_i, 1\le i\le n}}
T'(\bu_1,\dots,\bu_n|\bv_1,\dots,\bv_n)\Bigr\}
\Bigl|_{\substack{u_{i,a}\to u_{i,a}-\gamma,\\ 1\le a\le l_i, 1\le i\le n}}.
\end{gather*}

Proof of Theorem \ref{thm:main} is now complete.
\end{proof}

\section[The cases n=1,2]{The cases $\boldsymbol{n=1,2}$}\label{sec:n12}

All results discussed so far are valid also in the case $n=1,2$ with suitable modifications.
In this section we indicate the necessary changes.

\subsection[Case n=2]{Case $\boldsymbol {n=2}$}

Dealing with algebras $\E_n$ and $\cK_n$ for $n=2$, there are two points to be taken care of.

First, the structure functions are changed. We keep $g_{i,i}(z,w)$ unchanged
but replace $g_{i,j}(z,w)$ for $i\neq j$ with $g^{n\ge 3}_{i,j}(z,w)g^{n\ge 3}_{j,i}(z,w)$. Similar changes are due for
$G_{i,j}(x)$, $\lambda^0_{i,j}(x)$ and $\lambda_{i,j}(x)$.
Namely we use
\begin{gather*}
g_{i,1-i}(z,w)=(z-q_1w)(z-q_3w),
\\
G_{i,1-i}(x)=q^{-2}\frac{\big(1-q_1^{-1} x\big)\big(1-q_3^{-1} x\big)}{(1-q_1 x)(1-q_3 x)},
\\
\lambda^0_{i,1-i}(x)=\frac{1-q_1x}{1-x}\frac{1-q_3x}{1-x},
\\
\lambda_{i,1-i}(x)=x^\gamma \frac{1-q_1x}{1-x}\frac{1-q_3x}{1-x}\frac{\big(q_1^{-1}x,q_3^{-1}x;p\big)_\infty}{(q_1 x,q_3 x;p)_\infty}.
\end{gather*}
We set also $d_{i,i}=1$, $d_{i,1-i}=-1$.

Second, the Serre relations are modified.

The defining relations of $\E_2$ are the same as those for $n\ge 3$ except for the Serre relations
\begin{gather}
\Sym_{z_1,z_2,z_3}[e_i(z_1),[e_i(z_2),[e_i(z_3),e_{1-i}(w)]_{q^2}]]_{q^{-2}} =0,
\label{eeee}
\\
\Sym_{z_1,z_2,z_3}[f_i(z_1),[f_i(z_2),[f_i(z_3),f_{1-i}(w)]_{q^2}]]_{q^{-2}} =0.
\label{ffff}
\end{gather}
The defining relations of $\cK_2$ are \eqref{KK1}--\eqref{EE}, together with the Serre relations
\begin{gather}
\mathop{\Sym}_{z_1,z_2}\,
\bigl(q_1(z_1-q_3w)(z_2-q_3w)E_i(z_1)E_i(z_2)E_{1-i}(w)
 -\big(1+q_2^{-1}\big)(z_1-q_3w)(q_1z_2-w)\nn
\\ \qquad
{}\times E_i(z_1)E_{1-i}(w)E_i(z_2)
+q_3(q_1z_1-w)(q_1z_2-w)E_{1-i}(w)E_i(z_1)E_i(z_2)\bigr)\nn
\\ \qquad
{}=\mathop{\Sym}_{z_1,z_2}\,\biggl\{
q^{-3}w\bigg(\big(z_2-d^{-2}z_1\big)\frac{1-q_3^{-1}z_1/w}{1-q_3 z_1/w}
+\big(z_2-d^2z_1\big)\frac{1-q_1^{-1} w/z_2}{1-q_1 w/z_2}-z_2+q^2z_1\biggr)\nn
\\ \qquad\hphantom{=}
{}\times \delta\big(C^2z_1/z_2\big)K^-_i(z_1)E_{1-i}(w)K_i^+(z_2)\bigg\},\label{Serre_2}
\end{gather}
and the same relation with $q_1$ and $q_3$ interchanged.

At first glance, the quartic Serre relations \eqref{eeee} and~\eqref{ffff}
for $\E_2$ and the cubic Serre relations~\eqref{Serre_2} for $\cK_2$ look very different.
Actually the former are equivalent (under the rest of the relations)
to cubic relations which are the left-hand sides of \eqref{Serre_2} and their ``$f$-version''.
See~\cite[Lemma~2.1]{FJMM3} and remark after that.

With the above changes, formula for comodule structure (Theorem \ref{thm:comod}), formulas for
integrals of motion \eqref{IM1} and~\eqref{IMEM} and their commutativity (Theorem \ref{thm:main}) remain valid.

\subsection[Case n=1]{Case $\boldsymbol{n=1}$}

Algebra $\cK_1$ has been discussed in \cite{FJMV}. For reader's reference we mention the necessary changes.

We drop suffixes $i$, $j$ from structure functions and set
\begin{gather*}
g(z,w)=\prod_{s=1}^3(z-q_sw),\\
G(x)=\prod_{s=1}^3\frac{1-q_s^{-1}x}{1-q_sx},\\
\lambda^0(x)=\frac{1-C^2x}{1-x}\frac{1-C^{-2}x}{1-x}\prod_{s=1}^3 \frac{1-q_s x}{1-x},\\
\lambda(x)=\frac{1-C^2x}{1-x}\frac{1-C^{-2}x}{1-x}\prod_{s=1}^3 \frac{1-q_s x}{1-x}
\frac{\big(q_s^{-1}x;p\big)_\infty}{(q_s x;p)_\infty}.
\end{gather*}
We change also the normalization of generators of $\E_1$ slightly, so that
\begin{gather*}
\psi^\pm(z)=\exp\bigg(\sum_{\pm r>0}\kappa_r h_{\pm r}z^{\mp r}\bigg),\\
[e(z),f(w)]=\frac{1}{\kappa_1}(\delta\bigl(C{w}/{z}\bigr)\psi^+(w)
-\delta\bigl(C{z}/{w}\bigr)\psi^-(z)),
\end{gather*}
where $\kappa_r=\prod_{s=1}^3(1-q_s^r)$.

The Serre relations for $\E_1$ are
\begin{gather*}
\Sym_{z_1,z_2,z_3}z_2z_3^{-1}[e(z_1),[e(z_2),e(z_3)]] =0,
\\
\Sym_{z_1,z_2,z_3}z_2z_3^{-1}[f(z_1),[f(z_2),f(z_3)]] =0.
\end{gather*}
The Serre relation for $\cK_1$ reads
\begin{gather*}
\mathop{\mathrm{Sym}}_{z_1,z_2,z_3}\frac{z_2}{z_3}
[E(z_1),[E(z_2),E(z_3)]]
\\ \qquad
{}=\mathop{\mathrm{Sym}}_{z_1,z_2,z_3}\bigg\{X(z_1,z_2,z_3)
\frac{1}{q-q^{-1}}\delta\big(C^2 z_1/z_3\big)K^-(z_1)E(z_2)K^+(z_3)\bigg\},
\end{gather*}
where
\begin{align*}
X(z_1,z_2,z_3)={}&
\frac{(z_1+z_2)\big(z_3^2-z_1z_2\big)}{z_1z_2z_3}G(z_2/z_3)
+\frac{(z_2+z_3)\big(z_1^2-z_2z_3\big)}{z_1z_2z_3}G(z_1/z_2)
\\
&+\frac{(z_3+z_1)\big(z_2^2-z_3z_1\big)}{z_1z_2z_3}.
\end{align*}

The space of theta functions $X_1$ is defined to be that of constant functions $\C$.
Replacing $h_M(\mathbf{u}_1,\dots,\mathbf{u}_n;\vartheta)$ by
\begin{gather*}
h_M(\mathbf{u})=\prod_{1\le a<b\le M}
\frac{\theta(u_a-u_b)\theta(u_a-u_b-\gamma)}{\theta(u_a-u_b-\beta-\gamma/2)\theta(u_a-u_b+\beta-\gamma/2)},
\end{gather*}
we define integrals of motion by \eqref{IM1}. Then the commutativity Theorem \ref{thm:main} holds true.

\section{Fusion}\label{sec:fusion}

It is known \cite{FJMM3} that the completed quantum toroidal algebra $\tilde{\E}_n$ with parameters $q_1$, $q_2$, $q_3$
contains various subalgebras isomorphic to $\tilde{\E}_k$,
with $1\le k\le n-1$ and suitable parame\-ters~$\bar q_1$, $\bar q_2$, $\bar q_3$.
In this section we discuss an analogue of this construction for $\cK_n$.
Throughout this section we take $n\ge2$, and
fix an admissible $\cK_n$ module $V$, where $C$ acts as a scalar such that $C^2\neq \pm1,q^{\pm2}$.

Consider the following set of operators $\big\{\bar{E}_i(z), \bar{K}^\pm_i(z)\big\}_{0\le i\le n-2}$ acting on $V$:
\begin{gather}
\bar{E}_0(z)=\bar{c}\, E_{n-1,0}(q_1z,z),\qquad
\label{fus1}
\bar{K}^\pm_0(z)=K^\pm_{n-1}(q_1z)K^\pm_0(z),
\\
\bar{E}_i(z)=E_i\big(q_1^{\frac{i}{n-1}}z\big),
\qquad \bar{K}^\pm_i(z)=K^\pm_i\big(q_1^{\frac{i}{n-1}}z\big),\qquad 1\le i\le n-2.
\label{fus3}
\end{gather}
Here $\bar{c}=(qq_1)^{1/2}\big(1-q_1^{-1}\big) $ for $n\ge3$ and $\bar{c}=\big(1-q_1^{-1}\big)^2$ for $n=2$.

\begin{prop}
The operators \eqref{fus1} and~\eqref{fus3} together with the same $C$
 give an action of~$\cK_{n-1}$ on $V$ with parameters
\begin{gather*}
\bar{q}_1=q_1q_1^{\frac{1}{n-1}},\qquad \bar{q}_2=q_2,\qquad \bar{q}_3=q_3q_1^{-\frac{1}{n-1}}.
\end{gather*}
\end{prop}

\begin{proof}
Denote by $\bar{G}_{i,j}(x)$, $0\le i,j\le n-2$
the functions $G_{i,j}(x)$ for $\cK_{n-1}$ with the parameters given above.
Then we have
\begin{gather*}
\bar{G}_{i,i}(x)=G_{i,i}(x)^2G_{i+1,i}(q_1x)G_{i,i+1}\big(q_1^{-1}x\big),\qquad n\ge2,
\end{gather*}
and
\begin{gather*}
\bar{G}_{i,i\pm1}(x)=
\begin{cases}
G_{i,i\pm1}\big(q_1^{\pm\frac{1}{n-1}}x\big), & n\ge4,\\
G_{i,i\pm1}\big(q_1^{\pm\frac{1}{n-1}}x\big)G_{i\pm1,i}\big(q_1^{\mp\frac{1}{n-1}}x\big),& n=3.
\end{cases}
\end{gather*}
Using these one can verify the relations \eqref{KK1}--\eqref{EE} as rational functions.

For simplicity we assume $n\ge3$ below. The case $n=2$ can be treated with minor modifications.
Let us check the condition \eqref{quad-res} for $\bar E_0(z)$. We have
\begin{align*}
\bar E_0(z_1)\bar E_{0}(z_2)={}&E_{n-1,n-1,0,0}(q_1z_1,q_1z_2,z_1,z_2)
\bar c d \big(1-q_1^{-1}\big)^2\!
\frac{(1-z_2/z_1)^2}{(1-C^2z_2/z_1)^2}\frac{(1-z_2/z_1)^2}{(1-C^{-2}z_2/z_1)^2}
\\
&\times\frac{(1-z_2/z_1)^2}{\big(1-q^2z_2/z_1\big)^2}
\frac{\big(1-q_1^{-1}z_2/z_1\big)(1-q_1z_2/z_1)}{\big(1-q^{-2}z_2/z_1\big)(1-z_2/z_1)}.
\end{align*}
The wheel conditions \eqref{wheel1}--\eqref{wheel3}
ensure that the apparent double poles in the right-hand side are actually all simple.
To compute the residue at $z_2=C^2z_1$, it suffices to consider
\begin{gather*}
\frac{1}{1-C^{-2}z_2/z_1}E_{n-1,n-1,0,0}(q_1z_1,q_1z_2,z_1,z_2)\bigg|_{z_2=C^2z_1}
\\ \quad
{}=\!w \frac{\partial}{\partial w}E_{n-1,n-1,0,0}(q_1z_1,q_1z_2, w,z_2)
\bigg|_{\substack{w=z_1\\ z_2=C^2z_1}} \!\!\!\!\!
\!+\!E_{n-1,n-1,0,0}\big(q_1w,C^2q_1z_1,z_1,C^2z_1\big)\bigg|_{\substack{w=z_1\\ z_2=C^2z_1}}\!\!.
\end{gather*}
The second term in the right-hand side vanishes due to the wheel condition.
Substituting
\begin{align*}
E_{n-1,n-1,0,0}(q_1z_1,q_1z_2,w,z_2)
=&{}K_{n-1}^-(q_1z_1)E_{0,0}(w,q_1z_1)K_{n-1}^+\big(C^2q_1z_1\big)
\\
&\times \frac{1-w/z_1}{1\!-\!q_1^{-1}w/z_1}\frac{1-q^2C^{-2}w/z_1}{1\!-\!q_1^{-1}C^{-2}w/z_1}
\frac{-d^{-2}q^{-1}}{1-q_1^{-1}}\frac{1+C^2}{1-C^2}\frac{1-q^2C^2}{1\!-\!q_1^{-1}C^2},
\end{align*}
we arrive at
\begin{gather*}
\lambda^0_{0,0}(z_2/z_1)\bar E_0(z_1)\bar E_0(z_2)\bigl|_{z_2=C^2z_1}
=\frac{1}{q-q^{-1}}\frac{1+C^2}{1-C^2}\frac{1-q^2C^2}{1-C^2}\bar K_0^-(z_1)\bar K_0^+\big(C^2z_1\big).
\end{gather*}

{\sloppy
The wheel conditions for $\bar E_i(z)$ can be verified similarly.
As an example, consider $\bar E_{0,0,1}(z_1,z_2,w)$.
Up to an irrelevant factor it can be written as
\begin{gather}
\bar E_{0,0,1}(z_1,z_2,w) \propto
\frac{1}{\lambda^0_{0,0}(z_2/z_1)}E_{n-1,n-1,0,0,1}\big(q_1z_1,q_1z_2,z_1,z_2,q_1^{\frac{1}{n-1}}w\big).
\label{fus-wheel}
\end{gather}}
Since
\[
E_{n-1,n-1,0}(q_1z,q_1q_2z,z)=0
\qquad\text{and}\qquad \text E_{0,0,1}(z,q_2z,q_2q_3z)=0,
\]
$E_{n-1,n-1,0,0,1}\big(q_1z_1,q_1z_2,z_1,z_2,q_1^{-1}z_1\big)$ is divisible by $(z_2-q_2z_1)^2$.
Therefore \eqref{fus-wheel} vanishes at $z_2=q_2z_1$, $w=\bar q_1^{-1} z_1$.
Similarly, the condition $E_{n-1,n-1,0}(q_1z,q_1C^2z,z)=0$ and $E_{0,0,1}\big(z,C^2z,C^2q_3z\big)\allowbreak=0$ imply that
\eqref{fus-wheel} vanishes at $z_2=C^2z_1$, $w=C^2 \bar q_3 z_1$.
\end{proof}

Interchanging the roles of $q_1$ and $q_3$, one can equally well consider action of $\cK_{n-1}$ with parameters\vspace{-.5ex}
\begin{align*}
\bar{\bar{q}}_1=q_1q_3^{-\frac{1}{n-1}},\qquad \bar{\bar{q}}_2=q_2,\qquad \bar{\bar{q}}_3=q_3q_3^{\frac{1}{n-1}}.
\end{align*}
Note that in either case the values of $q$, $C$ are unchanged.

Iterating this construction we obtain in general
\begin{prop}\label{prop:fusion}
Let $1\le k\le n-1$. Then the following hold with an appropriate choice of constants
$\bar c_{n,k}$, $\bbar c_{n,n-k}$:
\begin{enumerate}\itemsep=0pt
\item[$(i)$] The operators\vspace{-.5ex}
\begin{gather*}
\bar{E}_0(z)=\bar{c}_{n,k}\, E_{k,k+1,\dots,n-1,0}\big(q_1^{n-k}z,q_1^{n-k-1}z,\dots,q_1z,z\big),
\\
\bar{K}^\pm_0(z)=K^\pm_k\big(q_1^{n-k}z\big)K^\pm_{k+1}\big(q_1^{n-k-1}z\big)\cdots K^\pm_{n-1}(q_1z)K^\pm_0(z),
\\
\bar{E}_i(z)=E_i\big(q_1^{\frac{n-k}{k}i}z\big),
\qquad \bar{K}^\pm_i(z)=K^\pm_i\big(q_1^{\frac{n-k}{k}i}z\big),\qquad 1\le i\le k-1,
\end{gather*}
together with the same $C$
give an action of $\cK_k$
on $V$ with parameters\vspace{-.5ex}
\begin{gather*}
\bar{q}_1=q_1q_1^{\frac{n-k}{k}},\qquad \bar{q}_2=q_2,\qquad \bar{q}_3=q_3q_1^{-\frac{n-k}{k}}.
\end{gather*}

\item[$(ii)$] The operators\vspace{-.5ex}
\begin{gather*}
\bbar{E}_0(z)=\bbar{c}_{n,n-k} E_{k,k-1,\dots,1,0}\big(q_3^{k}z,q_3^{k-1}z,\dots,q_3z,z\big),
\\
\bbar{K}^\pm_0(z)=K^\pm_k\big(q_3^{k}z\big)K^\pm_{k-1}\big(q_3^{k-1}z\big)\cdots K^\pm_{1}(q_3z)K^\pm_0(z),
\\
\bbar{E}_i(z)=E_{i+k}\big(q_3^{\frac{k}{n-k}i}z\big),
\qquad \bbar{K}^\pm_i(z)=K^\pm_{i+k}\big(q_3^{\frac{k}{n-k}i}z\big),\qquad 1\le i\le n-k-1,
\end{gather*}
together with the same $C$
give an action of $\cK_{n-k}$
on $V$ with parameters
\begin{gather*}
\bbar{q}_1=q_1q_3^{-\frac{k}{n-k}},\qquad \bbar{q}_2=q_2,\qquad \bbar{q}_3=q_3q_3^{\frac{k}{n-k}}.
\end{gather*}

\item[$(iii)$] The actions $(i)$ and $(ii)$ mutually commute.
\end{enumerate}
\end{prop}

\begin{proof}
We need only to prove the commutativity $(iii).$ For simplicity of presentation we assume $n\ge3$.
Evidently $\big\{\bar E_i(z),\bar K^\pm_i(z)\big\}_{1\le i\le k-1}$ commute with
$\big\{\bbar E_i(z),\bbar K^\pm_i(z)\big\}_{1\le i\le n-k-1}$.

To show that $\bbar K^+_i(z)$, $1\le i\le n-k-1$, commutes with $\bar E_0(w)$, it suffices to note that
\begin{align*}
K^+_j(z)E_{k,k+1,\dots,n-1,0}\big(q_1^{n-k}w,\dots,w\big)={}&E_{k,k+1,\dots,n-1,0}\big(q_1^{n-k}w,\dots,w\big)K^+_j(z)
\\
& \times G_{j,j-1}\big(q_1^{n-j+1}w/z\big)G_{j,j}\big(q_1^{n-j}w/z\big)
\\
& \times G_{j,j+1}\big(q_1^{n-j-1}w/z\big),\qquad k+1\le j\le n-1,
\end{align*}
and apply the identity\vspace{-.5ex}
\begin{gather*}
G_{i,i-1}(q_1 x)G_{i,i}(x)G_{i,i+1}\big(q_1^{-1}x\big)=1.
\end{gather*}
Commutativity of $\big\{\bbar E_i(z),\bbar K^\pm_i(z)\big\}_{1\le i\le n-k-1}$ with $\big\{\bar E_0(w),\bar K^\pm_0(w)\big\}$
(for $\bbar E_i(z)$ with $\bar E_0(w)$ as rational functions)
can be shown similarly.

We have further
\begin{align*}
\bbar K_0(z) \bar E_0(w) ={}&\bar E_0(w) \bbar K_0(z)
G_{k-1,k}\big(q_1^nq_2^kq_3 w/z\big)G_{k,k}\big(q_1^nq_2^k w/z\big)G_{k,k+1}\big(q_1^nq_2^kq_1^{-1} w/z\big)
\\
&\times G_{1,0}\big(q_3^{-1}w/z\big)G_{0,0}(w/z)G_{0,n-1}(q_1 w/z),
\end{align*}
so that we can apply
\begin{gather*}
G_{i-1,i}(q_3x)G_{i,i}(x)G_{i,i+1}\big(q_1^{-1}x\big)=1.
\end{gather*}
By the same token $\big\{\bbar E_0(z),\bbar K^\pm_0(z)\big\}$ commute with $\big\{\bar E_0(w),\bar K^\pm_0(w)\big\}$
(for $\bbar E_0(z)$ with $\bar E_0(w)$ as rational functions).

To finish the proof, it remains to examine the singularities of $\bbar E_i(z)\bar E_0(w)$.
For $1\le i\le n-k-1$, we use
\begin{align*}
E_j(z)E_{j-1,j,j+1}\big(q_1w,w,q_1^{-1}w\big)={}&E_{j-1,j,j,j+1}\big(q_1w,z,w,q_1^{-1}w\big)
\frac{1-q_1w/z}{1-q_2^{-1}w/z}\frac{1-w/z}{1-C^2w/z}
\\
&\times\frac{1-w/z}{1-q^2w/z}\frac{1-w/z}{1-C^{-2}w/z}\frac{1-q_1^{-1}w/z}{1-w/z},
\end{align*}
and observe that all poles are cancelled due to the wheel conditions. Finally for $i=0$ we have
\begin{gather*}
E_{k,k-1,\dots,1,0}\big(q_3^kz,\dots,z\big)E_{k,k+1,\dots,n-1,0}\big(q_1^{n-k}w,\dots,w\big)
\\ \qquad
{}=E_{0,0,1,\dots,k-1,k,k,k+1,\dots,n-1}\big(z,w,q_3z,\dots,q_3^{k-1}z,q_3z,q_1^{n-k}w,\dots,q_1w\big)
\\ \qquad\hphantom{=}
{}\times \bigl(\lambda^0_{k-1,k}\big(q_1^nq_2^k q_3w/z\big)\lambda^0_{k,k}\big(q_1^nq_2^kw/z\big)\lambda^0_{k,k+1}\big(q_1^nq_2^k q_1^{-1}w/z\big)\bigr)^{-1}
\\ \qquad\hphantom{=}
{}\times \bigl(\lambda^0_{0,n-1}(q_1 w/z)\lambda^0_{0,0}(w/z)\lambda^0_{1,0}\big(q_3^{-1}w/z\big)\bigr)^{-1},
\end{gather*}
so the cancellation works similarly as above.
\end{proof}

\appendix

\section{Theta function identities}\label{sec:identity}

As in Section \ref{subsec:IM} we use the space of theta functions $X_n(\mu_1,\dots,\mu_n)$
given by \eqref{theta-per1} and \eqref{theta-per2} for $n\ge2$.
For $n=1$, we define ${X}_1=\C$ to be the space of constant functions.
In this section we prove the following identity.

\begin{thm}\label{thm:KSidentity}
The following identities hold for all $n, M,N\ge1$, $\alpha,\gamma\in\C$ and $\vartheta_1,\vartheta_2\in {X}_n$:
\begin{gather}
\sum_{I_1,J_1}\cdots \sum_{I_n,J_n}
\vartheta_1(\ovu_{1,I_1},\dots,\ovu_{n,I_n})\vartheta_2(\ovu_{1,J_1},\dots,\ovu_{n,J_n})\nn
\\ \qquad
{}\times\prod_{i=1}^{n}
\frac{\prod_{\substack{a\in I_i\\ b\in J_{i+1}}}\theta(u_{i,a}-u_{i+1,b}-\alpha)
\prod_{\substack{b\in J_i\\ a\in I_{i+1}}}\theta(u_{i,b}-u_{i+1,a}-\alpha+\gamma)}
{\prod_{\substack{a\in I_i\\ b\in J_i}}\theta(u_{i,a}-u_{i,b})\theta(u_{i,a}-u_{i,b}-\gamma)}\nn
\\ \qquad
{}=\sum_{I_1,J_1}\cdots \sum_{I_n,J_n}
\vartheta_1(\ovu_{1,I_1},\dots,\ovu_{n,I_n})\vartheta_2(\ovu_{1,J_1},\dots,\ovu_{n,J_n})\nn
\\ \qquad\hphantom{=}
{}\times \prod_{i=1}^{n}
\frac{\prod_{\substack{b\in J_i\\ a\in I_{i+1}}}\theta(u_{i,b}-u_{i+1,a}-\alpha)
\prod_{\substack{a\in I_i\\ b\in J_{i+1}}}\theta(u_{i,a}-u_{i+1,b}-\alpha+\gamma)}
{\prod_{\substack{a\in I_i\\ b\in J_i}}\theta(u_{i,b}-u_{i,a})\theta(u_{i,b}-u_{i,a}-\gamma)}.
\label{main-idenity}
\end{gather}
Here we set $\ovu_{i,I}=\sum_{a\in I}u_{i,a}$ for a subset $I$ of $\{1,\dots,M+N\}$,
and the sum is taken over all partitions
\begin{gather*}
I_i\sqcup J_i=\{1,\dots,M+N\}\qquad \text{satisfying}\quad |I_i|=M, \quad |J_i|=N, \quad i=1,\dots,n.
\end{gather*}
\end{thm}

For the proof of theorem, we prepare several lemmas.

\begin{lem}\label{Theta Lem1}
If $\vartheta\in {X}_n(\mu_1,\dots,\mu_n)$, then
\begin{gather*}
\vartheta(u_1+u,\dots,u_n+u)=\vartheta(u_1,\dots,u_n),\qquad u\in\C.
\end{gather*}
\end{lem}

\begin{proof}
Let $\phi(u)=\vartheta(u_1+u,\dots,u_n+u)$.
Then $\phi(u)$ is entire and satisfies $\phi(u+1)=\phi(u)$, $\phi(u+\tau)=\phi(u)$. Hence it is a constant:
$\phi(u)=\phi(0)$.
\end{proof}

\begin{lem}\label{lem:sum}
Let $\phi(v_1,\dots,v_L)$ be an entire function with quasi-periodicity property
\begin{gather*}
\phi(v_1,\dots,v_i+1,\dots,v_L)=\phi(v_1,\dots,v_i,\dots,v_L),
\\
\phi(v_1,\dots,v_i+\tau,\dots,v_L)={\rm e}^{-2\pi {\rm i}m \sum_{i=1}^L v_i +k}\phi(v_1,\dots,v_i,\dots,v_L)
\end{gather*}
with some $m\in\Z_{\ge0}$ and $k\in\C$. Then it is a function of $\sum_{i=1}^Lv_i$.
\end{lem}

\begin{proof}
Fix $i\neq j$, and set $\varphi(u)=\phi(v_1,\dots,v_i+u,\dots,v_j-u,\dots,v_L)$. Then $\varphi(u)$ is an entire function
and satisfies $\varphi(u+1)=\varphi(u)$, $\varphi(u+\tau)=\varphi(u)$. Hence it is a constant: $\varphi(u)=\varphi(0)$.
Since $i$, $j$, $u$ are arbitrary, this proves the lemma.
\end{proof}

Denote both sides of the identity \eqref{main-idenity} by ${\rm LHS}_{n,M,N}$ and ${\rm RHS}_{n,M,N}$, respectively, and set $\Phi_{n,M,N}={\rm LHS}_{n,M,N}-{\rm RHS}_{n,M,N}$.
Where necessary we exhibit also their dependence on the variables
$\bu_i=(u_{i,1},\dots,u_{i,M+N})$ and
the choice of $\vartheta_1,\vartheta_2\in {X}_n$. We set $\Phi_{n,M,N}=0$ if $M=0$ or~$N=0$.

\begin{lem}
For $n\ge 2$, $\Phi_{n,M,N}$ has the quasi-periodicity in each variable $u_{i,a}$,
\begin{gather*}
\Phi_{n,M,N}(\dots, u_{i,a}+1,\dots)=\Phi_{n,M,N}(\dots, u_{i,a},\dots),\\
\Phi_{n,M,N}(\dots, u_{i,a}+\tau,\dots)={\rm e}^{-2\pi{\rm i}(2\bar u_i-\bar u_{i-1}-\bar u_{i+1}-\mu_i+\tau)}
\Phi_{n,M,N}(\dots, u_{i,a},\dots).
\end{gather*}
For $n=1$, $\Phi_{1,M,N}$ has periods $1$ and $\tau$.

Modulo $\Z\oplus\Z\tau$, the only poles of $\Phi_{n,M,N}$ are $u_{i,a}=u_{i,b}+\gamma$ with some $a$, $b$, $i$.
\end{lem}

\begin{proof}
The quasi-periodicity follows from \eqref{theta-per0}, \eqref{theta-per1} and \eqref{theta-per2}.
Since $\Phi_{n,M,N}$ is symmetric in each group of variables $\{u_{i,a}\}_{a=1,\dots,M+N}$, it cannot have a simple pole at
$u_{i,a}=u_{i,b}$.
\end{proof}

\begin{lem}\label{lem:it-res}
We have
\begin{gather*}
\Res_{u_{1,1}=u_{1,2}+\gamma}\cdots \Res_{u_{n,1}=u_{n,2}+\gamma}\,
\Phi_{n,M,N}(\bu_1,\dots,\bu_n|\vartheta_1,\vartheta_2)
\\ \qquad
{}=A^n
\frac{\prod_{i=1}^n\prod_{a=2}^{M+N}\theta(u_{i,a}-u_{i+1,2}-\alpha)\theta(u_{i,2}-u_{i+1,a}-\alpha+\gamma)}
{\prod_{i=1}^n\prod_{a=3}^{M+N}\theta(u_{i,a}-u_{i,2})\theta(u_{i,a}-u_{i,2}-\gamma)}
\\ \qquad\hphantom{=}
{}\times \Phi_{n,M-1,N-1}(\bu'_1,\dots,\bu'_n|\vartheta'_1,\vartheta'_2),
\end{gather*}
where
\begin{gather*}
\bu'_i=(u_{i,3},\dots,u_{i,M+N}),
\\
\vartheta'_s(v_1,\dots,v_n)=\vartheta_s(v_1+u_{1,2},\dots,v_n+u_{n,2}),\qquad s=1,2,
\end{gather*}
and $A=\theta(\gamma)^{-1}\Res_{u=0}\theta(u)^{-1}$.
\end{lem}

\begin{proof}
In the sum ${\rm LHS}_{n,M,N}$, the iterated residue is non-zero only for terms such that
$1\in I_i$ and $2\in J_i$ for all $i=1,\dots,n$.
Write $I_i=\{1\}\sqcup \bar I_i$ and $J_i=\{2\}\sqcup \bar J_i$.
Similarly, for ${\rm RHS}_{n,M,N}$ one must have
$2\in I_i$ and $1\in J_i$ for $i=1,\dots,n$. Write $I_i=\{2\}\sqcup \bar I_i$ and $J_i=\{1\}\sqcup \bar J_i$ in that case.

In either case, the functions $\vartheta_s$ become
\begin{gather*}
\vartheta_1(u_{1,2}+\bar u_{1,\bar I_1},\dots, u_{n,2}+\bar u_{n,\bar I_n})
\vartheta_2(u_{1,2}+\bar u_{1,\bar J_1},\dots, u_{n,2}+\bar u_{n,\bar J_n})
\end{gather*}
due to Lemma \ref{Theta Lem1}.

The rest is a simple computation.
\end{proof}

\begin{rem}\label{rem:gap}
Lemma \ref{lem:it-res} is proved in \cite{KS}. However the rest of the
argument necessary for completing the proof of Theorem \ref{thm:KSidentity} is missing there.
\end{rem}

\begin{lem}\label{lem:special}
Fix $i\in\{1,\dots,n\}$. Then the following hold:
\begin{gather*}
\Phi_{n,M,N}(\bu_1,\dots,\bu_n|\vartheta_1,\vartheta_2)
\bigl|_{\substack{u_{i+1,a}=u_{i,a}-\alpha\\ 1\le a\le M+N}}
\\ \qquad
{}=(-1)^{MN}\Phi_{n-1,M,N}(\bu_1,\dots,\bu_{i},\bu_{i+2},\dots,\bu_n|\vartheta''_1,\vartheta''_2),
\\
\Phi_{n,M,N}(\bu_1,\dots,\bu_n|\vartheta_1,\vartheta_2)
\bigl|_{\substack{u_{i+1,a}=u_{i,a}-\alpha+\gamma\\ 1\le a\le M+N}}
\\ \qquad
{}=(-1)^{MN}\Phi_{n-1,M,N}(\bu_1,\dots,\bu_{i},\bu_{i+2},\dots,\bu_n|\vartheta'''_1,\vartheta'''_2),
\end{gather*}
where
\begin{gather*}
\vartheta''_1(v_1,\dots,v_{n-1})=\vartheta_1(v_1,\dots,v_i,v_i-M\alpha,v_{i+1},\dots,v_{n-1}),
\\
\vartheta''_2(v_1,\dots,v_{n-1})=\vartheta_2(v_1,\dots,v_i,v_i-N\alpha,v_{i+1},\dots,v_{n-1}),
\end{gather*}
and $\vartheta'''_s$ is given by replacing $\alpha$ by $\alpha-\gamma$ in $\vartheta''_s$.
\end{lem}

\begin{proof}
Since the calculations are similar, we consider the first case.
Under this specialization, terms in ${\rm LHS}_{n,M,N}$ survive only when
$I_i\cap J_{i+1}=\varnothing$, which implies $I_i=I_{i+1}$ and $J_i=J_{i+1}$.
Then the factor
$\prod_{\substack{a\in I_i\\ b\in J_{i+1}}}\theta(u_{i,a}-u_{i+1,b}-\alpha)
\prod_{\substack{a\in I_{i+1}\\ b\in J_i}}\theta(u_{i,b}-u_{i+1,a}-\alpha+\gamma)$
from the numerator cancels
$\prod_{\substack{a\in I_{i+1}\\ b\in J_{i+1}}}\theta(u_{i+1,a}-u_{i+1,b})\theta(u_{i+1,a}-u_{i+1,b}-\gamma)$
from the denominator, thereby reducing~$n$ to $n-1$.
For ${\rm RHS}_{n,M,N}$ the situation is entirely similar.
\end{proof}

\begin{lem}[\cite{Ru}]
\label{lem:n=1}
The identity \eqref{main-idenity} holds for $n=1$.
\end{lem}

\begin{proof}
For $n=1$ we may assume $\vartheta_1=\vartheta_2=1$, so the identity takes the form
\begin{gather*}
\sum_{I,J}
\prod_{\substack{a\in I\\ b\in J}}
\frac{\theta(u_{a}-u_{b}-\alpha)\theta(u_{b}-u_{a}-\alpha+\gamma)}
{\theta(u_{a}-u_{b})\theta(u_{a}-u_{b}-\gamma)}
=\sum_{I,J}\prod_{\substack{a\in I\\ b\in J}}
\frac{\theta(u_{b}-u_{a}-\alpha)\theta(u_{a}-u_{b}-\alpha+\gamma)}
{\theta(u_{b}-u_{a})\theta(u_{b}-u_{a}-\gamma)}.
\end{gather*}
Lemma \ref{lem:it-res} together with the symmetry shows that
$\Phi_{1,M,N}$ has no pole at $u_a=u_b+\gamma$. Hence $\Phi_{1,M,N}$ is an elliptic function without poles, therefore it is a constant.
Upon specializing $u_a=a\alpha$, $a=1,\dots,M+N$, we obtain
\begin{align*}
{\rm LHS}_{1,M,N}&=(-1)^{MN}
\frac{\prod_{a=1}^M\theta((a-1)\alpha+\gamma)/\theta(a\alpha)\prod_{a=1}^N\theta((a-1)\alpha+\gamma)/\theta(a\alpha)}
{\prod_{a=1}^{M+N}\theta((a-1)\alpha+\gamma)/\theta(a\alpha)}
\\
&={\rm RHS}_{1,M,N},
\end{align*}
thereby proving $\Phi_{1,M,N}=0$.
\end{proof}

\begin{proof}[Proof of Theorem \ref{thm:KSidentity}]
We use induction on $n$. The base of induction is proved in Lem\-ma~\ref{lem:n=1}.
Assume that theorem is true for $n-1$.

For $i=1,\dots,n$ we set
\begin{gather*}
Z_i=\Res_{u_{i,2}=u_{i,1}-\gamma}\Res_{u_{i+1,2}=u_{i+1,1}-\gamma}
\cdots \Res_{u_{n,2}=u_{n,1}-\gamma}\, \Phi_{n,M,N}(\bu_1,\dots,\bu_n,\alpha,\gamma|\vartheta_1,\vartheta_2).
\end{gather*}
We show $Z_i=0$ by induction on $i$.
By Lemma \ref{lem:it-res} and the induction hypothesis $\Phi_{n-1,M,N}=0$, we have $Z_1=0$.
Suppose $Z_i=\Res_{u_{i,2}=u_{i,1}-\gamma}Z_{i+1}=0$. Then $Z_{i+1}$ has no poles in $\bu_i$.

Since $u_{i+1,2}=u_{i+1,1}-\gamma$, non-zero terms in $Z_{i+1}$ are such that
$1\in I_{i+1}$, $2\in J_{i+1}$ in {\rm LHS}. Then in the numerator we have factors
$\theta(u_{i,a}-u_{i+1,2}-\alpha)$, $a\in I_i$, and
$\theta(u_{i,b}-u_{i+1,1}-\alpha+\gamma)$, $b\in J_i$. Therefore each term has a factor
\begin{gather*}
\prod_{a=1}^{M+N}\theta(u_{i,a}-u_{i+1,2}-\alpha).
\end{gather*}
By the same argument, {\rm RHS}, and hence $Z_{i+1}$, is shown to have the same factor.
If $Z_{i+1}\neq 0$, this is a contradiction because the quasi-periodicity property of $Z_{i+1}$ implies that
it is a function of $\bar u_i=\sum_{a=1}^{M+N}u_{i,a}$, see Lemma~\ref{lem:sum}.
Hence we have $Z_{i+1}=0$.

It follows that $\Phi_{n,M,N}$ has no poles at $u_{n,a}=u_{n,b}-\gamma$.
Therefore, if non-zero, it is a function of the sum $\bar u_n=\sum_{a=1}^{M+N}u_{n,a}$.
On the other hand, by Lemma \ref{lem:special} and the induction hypothesis,
$\Phi_{n,M,N}$ has the following zeros with respect to $\bar u_{n}$:
\begin{gather*}
\bar u_{n}=\bar u_{1}+(M+N)\alpha,\qquad \bar u_{n}=\bar u_{1}+(M+N)(\alpha-\gamma),
\\
\bar u_{n}=\bar u_{n-1}-(M+N)\alpha,\qquad \bar u_{n}=\bar u_{n-1}-(M+N)(\alpha-\gamma).
\end{gather*}
From the quasi-periodicity, if $\Phi_{n,M,N}\neq0$, then it can have only two zeros modulo $\Z\oplus\Z\tau$,
which contradicts to the above zeros.

This completes the proof of Theorem \ref{thm:KSidentity}.
\end{proof}

\subsection*{Acknowledgments}

The study has been funded within the framework of the HSE University Basic Research Program.
MJ is partially supported by JSPS KAKENHI Grant Number JP19K03549.
EM is partially supported by grants from the Simons Foundation \#353831 and \#709444.

\pdfbookmark[1]{References}{ref}
\LastPageEnding

\end{document}